 \documentclass[11pt]{amsart} 
 
 \usepackage{amsxtra}
 \usepackage{amsthm}
 \usepackage{amscd}
 \usepackage{amsmath}
 \usepackage{amsthm}
 \usepackage{amsfonts}
 \usepackage{eucal}
 \usepackage{amssymb} 
 \usepackage[all]{xy}
 \usepackage{graphicx}
 \usepackage{epsfig}
 \usepackage{psfrag}
 \usepackage[curve]{xypic}
 \usepackage{xspace}
 \usepackage{layout}
 \usepackage{fancyhdr}
\newcommand{\C}{\mathbb{C}} 

\newcommand{\A}{\mathbb{A}}
\renewcommand{\P}{\mathbb{P}}
\newcommand{\Z}{\mathbb{Z}}
\newcommand{\Q}{\mathbb{Q}}
\newcommand{\G}{\mathbb{G}}
\newcommand{\iso}{\cong}

\newcommand{\mcal}[1]{\mathcal{#1}}

\newcommand{\ilim}{\mathop{\varprojlim}\limits}

\newtheorem{theorem}{Theorem}
 \newtheorem{lemma}[theorem]{Lemma}
 \newtheorem{proposition}[theorem]{Proposition}
 \newtheorem{corollary}[theorem]{Corollary}

 \theoremstyle{definition}
 \newtheorem{definition}[theorem]{Definition}
 \newtheorem{example}[theorem]{Example}

 \theoremstyle{remark}
 \newtheorem{remark}[theorem]{Remark}

\begin{document}


\title{Equivariant Algebraic Cobordism}
\author{Jeremiah Heller}
\email{heller@math.uni-wuppertal.de}
\author{Jos\'e Malag\'on-L\'opez}
\email{jmalagon@uottawa.ca}



\begin{abstract}
We construct an equivariant algebraic cobordism theory for schemes
with an action by a linear algebraic group over a field of
characteristic zero.  
\end{abstract}
\maketitle


\pagestyle{headings}


\section{Introduction}

In \cite{lev&mor-book} Levine-Morel construct a theory $\Omega_{*}$,
called algebraic cobordism, for schemes over a field of
characteristic zero. This theory is constructed so that its
restriction to smooth schemes is the universal oriented cohomology
theory. 
In this paper we extend their construction to the equivariant setting
and define an equivariant algebraic cobordism $\Omega^{G}_{*}$ for
schemes equipped with an action of a linear algebraic group $G$. 
Our construction is based on an idea used by Totaro \cite{tot-chow} 
to define Chow groups of a classifying space and Edidin-Graham  
\cite{edd&gra-eit} to define equivariant Chow groups.  
Their construction is motivated by the construction of equivariant
cohomology theories in algebraic topology using the Borel
construction. 
The Borel construction of a manifold $M$ with action by a Lie group
$G$ is the space $(M\times EG)/G$, where $EG$ is a contractible
$G$-space with free $G$-action. 
In the algebro-geometric setting, the classifying space of a linear
algebraic group cannot exist as a scheme. 
However reasonable approximations of the classifying space do exist
and the idea of the construction of the equivariant Chow groups is to
approximate the Borel construction. 
Chow groups have the property that the groups $\mathrm{CH}_{n}(X)$
vanish when $n>\dim X$. 
Because of this, in a fixed degree, there is a finite dimensional
approximation to the Borel construction which is a sufficiently good
approximation. 
On the other hand algebraic cobordism 
$\Omega_{n}(X)$ can be nonzero for arbitrarily large values of $n$. 
This leads us to define equivariant algebraic cobordism
$\Omega_{*}^{G}(X)$ in Section \ref{EAC} as the limit 
\[
\Omega_*^G (X) = 
  \varprojlim_{i}\Omega_* \left( X \times^G U_i \right)
\] 
over successively better approximations to the Borel construction. 
This involves making a choice of a sequence of approximations to the
Borel construction and we show in Theorem \ref{eacwelldef} that our
definition does not depend on the choice of sequence of
approximations. 
The rest of the section is devoted to some basic computations of the
resulting equivariant algebraic cobordism groups. 

In Section \ref{secfour} we establish the basic properties of
equivariant algebraic cobordism. 
It has all of the equivariant analogues of the basic properties
expected of an oriented Borel-Moore homology theory. 
These include pull-back maps for equivariant l.c.i.-morphisms,
push-forwards for equivariant projective morphisms and appropriate
compatibilities with pull-backs, homotopy invariance, projective
bundle formula, and a localization exact sequence. 
Additionally, equivariant algebraic cobordism has
properties that one would expect from an equivariant cohomology theory
for $G$-schemes, for example it is equipped with natural restriction
and induction maps.  

Suppose that $X\to X/G$ is a geometric quotient of $X$ by a proper
action of a reductive group $G$. 
There is in this case a natural map 
$\Omega_{*}(X/G) \to \Omega_{*+\textrm{dim}(G)}^{G}(X$) 
relating the algebraic cobordism of $X/G$ to the  equivariant
algebraic cobordism of $X$.  
In Theorem \ref{propact} this map is shown to be an isomorphism
with rational coefficients. 
As a consequence we obtain a naturally defined ring structure on the
rational algebraic cobordism of the quotient $X/G$ of a smooth scheme.  

If $G$ is split reductive, with maximal torus $T$, Weyl group $W$, and
$X$ is a smooth $G$-scheme then we show in Theorem \ref{the-loc} that
$\Omega^{*}_{G}(X)_{\Q}\cong\Omega^{*}_{T}(X)^{W}_{\Q}$. 
If $G$ is special then this is an injection integrally, and under
certain assumptions (see Remark \ref{rem-gsp}) can be shown to be
an isomorphism. 
In order to prove this result we require a cobordism version of the
Leray-Hirsch type theorem for  smooth projective fibrations $p:X\to Y$
with cellular fiber $F$. 
This is done in Proposition \ref{lhrat} where we show that 
$\Omega^{*}(X)_{\Q} =
\Omega^{*}(Y)_{\Q}\otimes_{\Omega^{*}(k)_{\Q}}\Omega^{*}(F)_{\Q}$ 
and it is an integral isomorphism if the fibration is Zariski-locally
trivial. 

We finish in Section \ref{ROT} with a brief discussion of oriented
equivariant Borel-Moore homology theories.
We focus on the examples of such theories that arise from Chow groups
and algebraic $K$-theory, and its relations with equivariant Chow
groups \cite{edd&gra-eit} and equivariant $K$-theory \cite{tho-ekt}.

Independently, another definition of equivariant algebraic cobordism 
has been given in \cite{Krishna}.
The version there is based on Desphande's definition of algebraic
cobordism for classifying spaces \cite{des-accs}, which involves
taking an inverse limit over approximations to the Borel construction
as well as quotients by a niveau filtration. 
That definition gives an equivariant algebraic cobordism theory
isomorphic to the one we have defined here, see Remark
\ref{krishcompare}.    

\vspace{.2cm}
{\bf Notation and Conventions.}
Throughout $k$ will denote a field of characteristic zero.
Let {\bf Sch}$_k$ denote the category whose objects are  separated,
quasi-projective  schemes of finite type over $k$, and let 
{\bf Sm}$_k$ be the full-subcategory of smooth quasi-projective
$k$-schemes.  
If $A = \lim_{i}A_{i}$ is an inverse limit of groups, we
write $A_{\Q} =\lim_{i}(A_{i}\otimes \Q)$. 

\vspace{.2cm}
{\bf Acknowledgements.}
We thank J.F. Jardine for his kind support during the preparation of
this work. We are grateful to the referee for helpful comments and
suggestions.  
The second author benefited from discussions with R. Gonzales, for
which he is grateful.


\section{Preliminaries}\label{preliminaries}

\subsection{Algebraic Cobordism}\label{OCT&AC}

In this section we recall definitions and terminology from 
\cite{lev&mor-book}.

\subsubsection{}\label{OBM}

Let {\bf Ab}$_*$ be the category of graded Abelian groups.
Let $\mathbf{Sch}_k'$ be the subcategory of ${\mathbf{Sch}}_k$ which
has the same objects and the morphisms are the projective morphisms. 
An \emph{oriented Borel-Moore homology theory (OBM)} on $\mathbf{Sch}_k$ 
consists of an additive functor $A_* : {\bf Sch}_k' \to {\bf Ab}_*$, 
pull-backs maps $f^*: A_* (X) \to A_{* + d} (Y)$ for each
l.c.i. morphism $f: Y \to X$ of relative dimension $d$,
and an associative and commutative external product 
$A_* (X) \otimes A_* (Y) \to A_* (X \times Y)$, 
$u \otimes v \mapsto u \times v$, with unit $1 \in A_0 (k)$.

These data satisfy certain axioms:
functoriality of pull-back maps, 
compatibility of push-forwards maps and pull-backs maps in transverse
Cartesian squares, compatibility of variances with product of schemes,
and 
\begin{enumerate}
\item[(EH)] \textit{Extended homotopy.} Let $E$ be a vector bundle of
  rank $r$ over $X$ in {\bf Sch}$_k$. 
  Let $p: V \to X$ be a $E$-torsor. 
  Then the induced morphism $p^* :A_* (X) \to A_{* + r} (V)$ is an
  isomorphism.    
 
\item[(PB)] \textit{Projective Bundle Formula.} Let $E$ be a vector
  bundle of rank $r + 1$ over $X$ in {\bf Sch}$_k$.
  Let $q : \mathbb{P} (E) \to X$ be the associated projective space. 
  Consider the canonical quotient line bundle 
  $\mathcal{O}(1) \to \mathbb{P} (E)$ with zero section 
  $s : \mathbb{P} (E) \to \mathcal{O} (1)$. 
  Consider the operator 
  $\xi : A_* (\mathbb{P} (E)) \to A_{*-1} ( \mathbb{P}(E))$, 
  $\eta \mapsto s^* ( s_* (\eta))$.
  For $i = 0 , \ldots, r$, let
\[
  \xymatrix{
            \xi^{(i)} : A_{* + i - r} (X) \ar[r]^{q^*} &
            A_{* + i} (\mathbb{P} (E)) \ar[r]^{\xi^i} &
            A_* (\mathbb{P} (E)).
           }
\]
  Then $\sum_{i=0}^r \xi^{(i)} : \oplus_{i=0}^r A_{* + i - r} (X)
  \to A_* (\mathbb{P} (E))$ is an isomorphism.

\item[(CD)] \textit{Cellular Decomposition.} Let $r,N > 0$. 
  Consider $W = \mathbb{P}^N \times \cdots \times \mathbb{P}^N$ with
  $r$ factors, and let $p_i : W \to \mathbb{P}^N$ be the $i$-th
  projection. 
  Let $X_0, \ldots, X_N$ be the standard homogeneous coordinates on
  $\mathbb{P}^N$. 
  Let $n_1, \ldots, n_r$ be non-negative integers, and let 
  $i : E \to W$ be the subscheme defined by 
  $\prod_{i=1}^r p_i^* (X_N)^{n_i} = 0$. 
  Then $i_* : A_* (E) \to A_* (W)$ is injective.
\end{enumerate}

A \emph{morphism} of OBMs is a natural transformation of functors.
Given an OBM $A_*$, the \emph{Chern class endomorphism} 
$\tilde{\mathrm{c}}_1 (L) : A_* (X) \to A_{*-1} (X)$ associated to a
line bundle $L\to X$ is given by 
$\tilde{\mathrm{c}}_1 (L) = s^{*}s_{*}$ where $s:X\to L$ is the
zero-section. 
There is an infinite series $F_A (u,v)$ with coefficients in $A_*
(k)$, such that for any line bundles $L,M$ over $X$, 
$\tilde{\mathrm{c}}_1 (L \otimes M) = 
F_A \left( \tilde{\mathrm{c}}_1 (L) , \tilde{\mathrm{c}}_1(M)
\right)$.

\subsubsection{}\label{OCT}

Let {\bf CRng}$^*$ be the category whose objects are 
commutative graded rings with unit and maps the ring morphisms.
An \emph{oriented cohomology theory (OCT)} on {\bf Sm}$_k$ consists of
an additive functor 
$A^* : \left( {\bf Sm}_k \right)^{{\rm op}} \to {\bf CRng}^*$, 
endowed with morphism of graded $A^*(X)$-modules
$f_* : A^* (Y) \longrightarrow A^{* + d} (X)$ for each
projective morphism $f : Y \to X$ of relative codimension $d$
satisfying certain axioms: 
functoriality of push-forwards maps, 
compatibility of variances in transverse Cartesian squares,   
as well as the analogues of the extended homotopy axiom (EH) and the
projective bundle formula (PB).   

A \emph{morphism} of OCTs is a natural transformation of
contravariant functors that also commutes with the push-forwards maps. 
Setting $\mathrm{c}_1 (L) := \tilde{\mathrm{c}}_1 (L) (1)$ we obtain a 
\emph{first Chern class} element.
For every OCT $A^*$ the pair $(A^*(k) , F_A )$ defines a formal group
law (\cite[Lemma [1.1.3]{lev&mor-book}).
Thus, there is a unique ring morphism $\theta_A : \mathbb{L} 
\to A^* (k)$ classifying $(A^*(k) , F_A )$.

\subsubsection{}\label{OBM&OCT}

Let $A_*$ be an OBM and let $X$ be in $\mathbf{Sm}_k$ be of pure
dimension $d$. 
Set $A^n (X) : = A_{d-n} (X)$ and extend this definition to any smooth
scheme by additivity over the connected components. 
Let $\delta_X : X \to X \times X$ be the diagonal morphism. 
The product $a \cup_X b := \delta^* (a \times b)$ makes $A^*(X)$ a
commutative ring with unit $1_X := p^* (1)$, where  
$p : X \to \mathrm{Spec} (k)$ is the structural morphism.  

\begin{proposition}[Levine-Morel]
The correspondence $A_* \mapsto A^*$ gives an equivalence of the
category of OBMs on $\mathbf{Sm}_k$  with the category of OCTs on
$\mathbf{Sm}_k$.
\end{proposition}

\subsubsection{}\label{AC}

We summerize the main properties of algebraic cobordism in the following

\begin{theorem}[Levine-Morel]
Let $k$ be a field that admits resolution of singularities.
\begin{enumerate}
\item There is a universal oriented Borel-Moore homology theory on 
  {\bf Sch}$_k$, $X \mapsto \Omega_* (X)$, called 
  \emph{algebraic cobordism}.
  The restriction of $\Omega_*$ to {\bf Sm}$_k$
  yields the universal oriented cohomology theory $\Omega^*$ on 
  {\bf Sm}$_k$.

\item Let $F_{\Omega}$ be the formal group law associated to
  $\Omega^*$. 
  Then the morphism
  $\theta_{\Omega}:  \mathbb{L} \longrightarrow \Omega^* (k)$ 
  classifying $F_{\Omega}$ is an isomorphism.

\item (Localization Sequences) If $\imath : Z \to X$ is a
  closed immersion and 
  $j: U := X - Z\to X$ is the open complement, then the following
  sequence is exact:   
\[
  \xymatrix@1{
              \Omega_* (Z) \ar[r]^{\imath_*}
              & 
	      \Omega_* (X) \ar[r]^{j^*} &
	      \Omega_* (U) \ar[r] &
	      0
	     }.
\tag{LS}
\]
\end{enumerate}
\end{theorem}
As an Abelian group, $\Omega_* (X)$ is generated by  isomorphism
classes ${\rm M}_n (X)$ of projective morphisms $Y \to X$, with $Y$ in
{\bf Sm}$_k$ and irreducible of dimension $n$.

\subsubsection{}\label{UOCT}

Given any formal group law $(R,F_R)$, the canonical morphism 
$\Omega^* (k) \to R$ induces an OCT $\, \Omega^* \otimes_{\mathbb{L}} R$,
defined by $X\mapsto \Omega^{*}(X)\otimes_{\mathbb{L}} R$, which is
universal for OCTs with formal group law $(R,F_R)$.  

\begin{example}\label{exa-uni-oct-fgl}
\begin{enumerate}
\item Let $\Omega^*_+ := \Omega^* \otimes_{\mathbb{L}} \mathbb{Z}$ be
  the OCT classifying the additive formal group law 
  $(\mathbb{Z} , u+v)$. 
  Then we have a canonical morphism $\Omega^*_+ \to {\rm CH}^*$.

\item Let $\Omega^*_{\times} : = \Omega^* \otimes_{\mathbb{L}}
  \mathbb{Z} [\beta, \beta^{-1}]$ be the OCT 
  classifying the multiplicative periodic formal group law 
  $(\mathbb{Z} [\beta, \beta^{-1}] , u + v - uv \beta)$. 
  Then we have a canonical morphism 
  $\Omega^*_{\times} \to {\rm K}^0 [\beta, \beta^{-1}]$.   
\end{enumerate}
\end{example}

\begin{theorem}[Levine-Morel]
Let $k$ be a field of characteristic zero. Then, the canonical
morphisms $\Omega^*_+ \to {\rm CH}^*$ and 
$\Omega^*_{\times} \to {\rm K}^0 [\beta, \beta^{-1}]$ are
isomorphisms of OCTs. 
\end{theorem}

\subsubsection{}\label{TOCT}

We recall the technique of twisting an OBM 
\cite[\S 7.4.1]{lev&mor-book}.
Fix an OBM $A_*$.
Let $\tau = (\tau_i)_{i \geq 0}$ be a sequence where $\tau_i \in 
A^{-i} (k)$ for all $i$ and $\tau_0$ is a unit. 
The $\tau$-\emph{inverse Todd class operator} of a line bundle 
$L \to X$ is defined as 
$\widetilde{{\rm Td}}^{-1}_{\tau} (L) := \sum_{i \geq 0} \tau_i \, 
\tilde{\mathrm{c}}_1^i (L)$.  This definition is extended to any vector bundle via the splitting.

The \emph{twisted OBM} $A_*^{(\tau)}$ is the OBM given by the data:
set $A_*^{(\tau)} (X) := A_*(X)$ for any $X$ in {\bf Sch}$_k$,
set $f_*^{(\tau)} := f_*$ for any projective morphism $f$,
and for any l.c.i. morphism $f: Y \to X$, 
given $x \in  A_*^{(\tau)} (X)$, set
$f^*_{(\tau)} (x) := \widetilde{\mathrm{Td}}^{-1} (N_f) \cdot f^* (x)$,
where $N_{f}$ denotes the virtual normal bundle of $f$. 

\begin{example}\label{exa-tt}
Let $\mathbb{Z} [ \mathbf{t}] := \mathbb{Z} [ t_i \mid i \geq 1]$, 
where $t_i$ is a variable of degree $-i$. Set $t_0 = 1$ and
define ${\bf t} = (t_0, t_1, \ldots)$.
By means of the ``universal exponential'' it can be shown that 
that there is an isomorphism  
$\log :\mathbb{Q}[\mathbf{t}] \to \mathbb{L}_\mathbb{Q}$ of 
$\mathbb{Q}$-algebras. 
Levine and Morel used this to show 
\cite[Theorem 4.1.28, Theorem 4.5.1]{lev&mor-book}
that the canonical map  
$
\Theta_{\exp}:(\Omega_*)_\mathbb{Q} 
\to
{\rm CH}_*\otimes_{\mathbb{Z}} \mathbb{Q} [\mathbf{t}]^{({\bf t})}
$ 
is an isomorphism of Borel-Moore homology theories. 

\end{example}

\subsubsection{}\label{accv-propo2}

For later applications we need a cobordism version of the
Leray-Hirsch theorem for smooth projective fibrations with cellular
fiber, proved for Chow groups in \cite[Lemma 2.8]{ES:quot}. 

A morphism $X\to Y$ is said to be a 
\emph{locally isotrivial fibration} 
with fiber $F$ provided that every $y\in Y$ has an open neighborhood $U$
admitting a finite \'etale morphism $f: V\to U$ such that 
$f^{-1}X =  V\times F$, and the map $f^{-1} X \to V$ agrees with the
projection $V\times F \to V$.  

A smooth variety $F$ is said to be
\emph{cellular} provided it has a filtration 
\begin{equation*}\label{equ-cel-fil}
  \varnothing = F_{0}
\subseteq
  F_1 
\subseteq  \ldots  \subseteq 
  F_{N-1} 
\subseteq 
  F_N = F,
\end{equation*}
where $\alpha_{i}:F_{i} \hookrightarrow F_{i+1}$ is a closed embedding and
$F_{i}- F_{i-1} \iso \A^{n_i}$ for all $i$.  
Since a cellular variety is built from affine spaces its cobordism is
easy to describe as an $\Omega^{*}(k)$-module.  

\begin{proposition}(\cite[Theorem 2.5]{jens&val-schubert})\label{invbase}
Let $F$ be a smooth, projective cellular variety with filtration
$\{ F_i \}$ as above. 
Then  $\Omega_{*}(F)$ is a free module over $\Omega_{*}(k)$. 
A basis is given by $[\widetilde{F_{i}}\to F_{i}\subseteq F_{m}]$
where $\widetilde{F_{i}}\to F_{i}$ is a resolution of singularities
for each $i\leq m$. 
If $L/k$ be a field extension, then the the map induced by base change
$\Omega_{*}(F) \to \Omega_{*} (F_{L})$ is an 
isomorphism.    
\end{proposition}

\begin{proposition}\label{lhrat}
Suppose that $F$ is a smooth, projective cellular variety and $p:X \to
Y$ is a projective, locally isotrivial fibration between smooth
varieties with fiber $F$. Let $\iota_{y}:F_{y}\to X$ be the
inclusion of a fiber over a  point $y\in Y$.  
\begin{enumerate}
\item The morphism
  $\iota_{y}^{*}:\Omega^{*}(X)_{\Q}\to\Omega^{*}(F_{y})_{\Q}$ is
  surjective for any point $y\in Y$. Moreover, under the
  isomorphism $\Omega^{*}(F_{y}) {\cong} \Omega^{*}(F)$  induced by base change along $k\to k(y)$, we have  that $\iota_{y}^{*}=\iota_{y'}^{*}$ for any two points
  $y,y'$ in the same connected component. 
  
\item Let $e_{i}\in \Omega^{*}(X)_{\Q}$ be elements such that
  $\eta_{i} = \iota_{y}^{*}e_{i}$ form a basis for
  $\Omega^{*}(F_{y})_{\Q}$ as an $\Omega^* (k)_{\Q}$-module.  
  Define the morphism
\begin{equation*}
\Psi:  \Omega^{*}(Y)_{\Q} \otimes_{\Omega^{*}(k)_{\Q}} \Omega^{*}(F)_{\Q} 
\to
  \Omega^{*}(X)_{\Q},
\end{equation*}
by 
$\Psi(\sum\alpha_{i} \otimes \eta_{i} ) = \sum p^*(\alpha_{i}) \cup e_{i}$. 
Then $\Psi$ is an isomorphism of $\Omega^{*}(k)_{\Q}$-modules. 

\item If $p:X\to Y$ is Zariski-locally trivial then statements (1) and
  (2) hold integrally. 
\end{enumerate}
\end{proposition}
\begin{proof}
\begin{enumerate}
 \item We may assume $Y$ is connected. Let $y\in Y$ be any point. We show that $\iota_{y}^{*}$ is surjective. 
  Let $j:U \to Y$ be an open neighborhood of $y$ which admits a finite \'etale morphism
  $f:V\to U$  over which the fibration trivializes. 
  Write $f_{V}:V\times F \to U\times_{Y} X$ for the induced morphism.
  Let $\beta_{i}\in \Omega^{*}(F)_{\Q}$ be a basis (as an 
  $\Omega^{*}(k)_{\Q}$-module)  and consider
  $(f_{V})_{*}(1_{V}\times\beta_{i}) \in \Omega^{*}(U\times_{Y}X)_{\Q}$. 
  Now  $j_{U}: U \times_{Y} X \to X$ is open and take
  $e_{i}\in \Omega^{*}(X)_{\Q}$ such that 
  $j_{U}^{*} (e_{i}) = (f_{V})_{*}(1_{V}\times\beta_{i})$.
 Then  we have that
  $\iota_{y}^{*} (e_{i}) = \deg(f)\beta_{i}\in\Omega^{*}(F_{y})_{\Q}$ and
  thus forms a basis.

  Now we show that $\iota^{*}_{y}=\iota^{*}_{\eta}$ where $\eta$ is the generic point of $Y$, 
  and so of $U$ as well. 
  Let $\eta'$ be the generic point of $V$  and $y'\in V$ a point such that $f(y')=y$. 
  The statement follows from consideration of the following commutative diagram, 
  where the bottom vertical arrow is induced by $V\times_{U}X\iso V\times F \to F$
\begin{equation*}
  \xymatrix@-0.7pc{
            \Omega^{*}(F_{y}) \ar[d]_{\cong} & \Omega^{*}(X) \ar[r]\ar[l]\ar[d] & 
            \Omega^{*}(F_{\eta})\ar[d]^{\cong} \\
            \Omega^{*}(F_{y'}) & \Omega^{*}(V\times_{U}X) \ar[r]\ar[l] & \Omega^{*}(F_{\eta'}) .\\
             & \Omega^{*}(F)\ar[ul]^{\cong}\ar[ur]_{\cong} \ar[u]& 
           }
\end{equation*}
\item We begin by showing that $\Psi$ is surjective. 
  Let $f:W \to Y$ be a morphism of $k$-schemes and write 
  $X_{W} = W \times_{Y} X$. 
  Let  $p_{W}:X_{W}\to W$ and $f_{W}:X_{W}\to X$ be the morphism
  obtained by base-change. 
  The morphisms $p_{W}$ and $f_{W}$ induce 
  $\langle p_{W},f_{W}\rangle:X_{W}\to W\times X$. 
  This is an l.c.i.-morphism since both $X_{W}$ and $W\times X$ are
  smooth over $W$.  
  Define the morphism of $\Omega^{*}(k)$-modules  
\[
  \Psi_{f}:\Omega_{*}(W) \otimes_{\Omega^{*}(k)}\Omega^{*}(F) \to
  \Omega_{*}(X_{W}),
\quad
  \sum w_{i} \otimes \eta_{i} \mapsto \sum \langle p_{W},f_{W}
  \rangle^{*} (w_{i}\times e_{i}) .
\] 
  If $f$ is smooth then $\Psi_{f} \left( \sum w_{i}
  \otimes \eta_{i} \right)= \sum f_{W}^{*}(w_{i}) \cup  e_{i}$,
  in particular $\Psi_{id_{Y}}$ is the map of the proposition. 
  We proceed by induction on the dimension of $W$ to show that
  $\Psi_{f}$ is surjective. 
  The zero-dimensional case follows directly from the definition of
  $\Psi_{f}$. 
  Suppose that $\Psi_{f'}$ is surjective where $f':W'\to Y$ with 
  $\dim W' < \dim W$. 
  Let $j:U \to W$ be an open over which the fibration 
  $X_{W} \to W$ becomes isotrivial and let $i:Z \to W$ be the closed
  complement.  
  Consider the comparison of exact sequences (where the tensor product
  is over $\Omega^{*}(k)_{\Q}$)
\begin{equation*}
\xymatrix{
          \Omega_{*}(Z)_{\Q} \otimes_{} \Omega^{*}(F)_{\Q} 
            \ar[r]^{i_{*}\otimes\,\mathrm{id}}\ar[d]^{\Psi_{f_{|Z}}} &
          \Omega_{*}(W)_{\Q} \otimes_{} \Omega^{*}(F)_{\Q} 
            \ar[d]^{\Psi_{f}}\ar[r]^{j^{*}\otimes\,\mathrm{id}} & 
          \Omega_{*}(U)_{\Q}\otimes_{}\Omega^{*}(F)_{\Q}
            \ar[r]\ar[d]^{\Psi_{U}} & 0\\
          \Omega_{*}(X_{Z})_{\Q} \ar[r]^{i'_{*}} & \Omega_{*}(X_{W})_{\Q}
            \ar[r]^{(j')^{*}} & 
          \Omega_{*}(X_{U})_{\Q} 
            \ar[r] & 
           0.
}
\end{equation*}
  This diagram commutes and the left-hand vertical map is surjective by
  induction. 
  It suffices to conclude that the right-hand vertical map is
  surjective. 
  Let $g:V\to U$ be a finite, \'etale morphism over which $X_{U} \to U$
  becomes trivial. We have the commutative square  
\begin{equation*}
\xymatrix{
          \Omega_{*}(V)_{\Q} \otimes_{} \Omega^{*}(F)_{\Q} 
            \ar[r]^{g_{*}\otimes\,\mathrm{id}}\ar[d]^{\Psi_{V}} &
          \Omega_{*}(U)_{\Q} \otimes_{} \Omega^{*}(F)_{\Q} 
            \ar[d]^{\Psi_{U}}\\
          \Omega_{*}(V\times F)_{\Q} 
            \ar[r]^{g'_{*}} & 
          \Omega_{*}(X_{U})_{\Q}.
}
\end{equation*}
  The morphism $g'_{*}$ is surjective since $g'_{*}(g')^{*}$ is
  multiplication by $\deg(g')$. 
  It only remains to see that $\Psi_{V}$ is surjective as
  well in order to conclude that $\Psi$ is surjective. 
  First note that $\Psi_{V}$ is the morphism
  $\Omega_{*}(V)\otimes_{\Omega^{*}(k)}\Omega_{*}(F) \to
  \Omega_{*}(V\times F)$ 
  induced by external product $\alpha \otimes \beta \mapsto
  \alpha\times \beta$.  
  Let $\{F_{i}\}$ be a filtration of $F$ as above and consider the
  commutative diagram with exact rows 
 \begin{equation*}
   \xymatrix{
             \Omega_{*}(V)\otimes \Omega_*(F_{k})
                \ar[r] \ar[d] &
 	    \Omega_{*}(V)\otimes \Omega_*(F_{k+1})
                \ar[r] \ar[d]&
 	    \Omega_{*}(V)\otimes \Omega_*(F_{k+1} -
                F_{k}) \ar[r] \ar[d] &
 	    0 \\ 
 	    \Omega_*(V\times F_{k}) \ar[r] & 
 	    \Omega_*(V\times F_{k+1}) \ar[r] &
 	    \Omega_*(V\times (F_{k+1} - F_{k})) \ar[r] & 
 	    0 . 
            }
\end{equation*}
  The right hand map is always a surjection and the left hand map is a
  surjection by induction and therefore the middle map is also a
  surjection.
 
  To show injectivity we first observe that the surjectivity result in
  the previous paragraph implies a decomposition of the class of the
  diagonal $\Delta_{*}(1_{X})\in \Omega^{*}(X\times_{Y}X)_{\Q}$. 
  Write $\pi_{k}:X\times_{Y}X\to X$ for the projection to the $k$-th
  factor. 
  The map $\pi_{2}:X\times_{Y} X\to X$ is a locally isotrivial fibration
  with fiber $F$ and the elements $\pi_{1}^{*} (e_{i})$ restrict to a basis
  of the fiber. 
  Therefore $\Omega^{*}(X\times_{Y}X)_{\Q}$ is generated by the elements
  $\pi_{1}^{*} (e_{i})$ as an $\Omega^{*}(X)_{\Q}$-module, where 
  $\Omega^{*}(X\times_{Y}X)_{\Q}$ is viewed as an
  $\Omega^{*}(X)_{\Q}$-module via $\pi_{2}^{*}$. 

  This means that there are elements $\alpha_{i}\in\Omega^{*}(X)_{\Q}$
  so that 
  $\Delta_{*}(1_{X}) = \sum \pi_{1}^{*} (e_{i}) \cup \pi_{2}^{*} (\alpha_{i})$
  where $\pi_{i}: X \times_{Y} X \to X$ are the projections. 
  Define a morphism of $\Omega^{*}(k)_{\Q}$-modules
  $\rho:\Omega^{*}(X)_{\Q}\to
  \Omega^{*}(Y)_{\Q}\otimes_{\Omega^{*}(k)_{\Q}}\Omega^{*}(F)_{\Q}$ by
  $\rho(x) = \sum p_{*}(x\cup \alpha_{i})\otimes \eta_{i}$. 
  Then injectivity of $\Psi$ follows from the equalities
\begin{align*}
  \Psi(\rho(x)) &= 
  \sum p^{*}p_{*}(x\cup \alpha_{i})\cup e_{i} 
= 
  \sum (\pi_{1})_{*} \pi_{2}^{*} (x\cup \alpha_{i})\cup e_{i}
  \\
&= (\pi_{1})_{*} \left( \sum\pi_{2}^{*} (x) \cup \pi_{2}^{*}
  (\alpha_{i}) \cup \pi_{1}^{*} (e_{i}) \right) = 
  (\pi_{1})_{*} \left( \pi_{2}^{*} (x) \cup \Delta_{*}(1_{X}) \right) \\
&=
  (\pi_{1})_{*}\Delta_{*} \left(\Delta^{*}\pi_{2}^{*} (x) \right) =  (\mathrm{id}_{X})_{*}(\mathrm{id}_{X})^{*} (x) 
 = x. 
\end{align*}

\item If $p:X\to Y$ is Zariski-locally trivial then $V=U$ in the proof
  of (1), so the argument given works integrally. Similarly for (2).
\end{enumerate}
\end{proof}

\subsection{Algebraic Groups and Algebraic Quotients}\label{AG&AQ}

\subsubsection{}

Given a linear algebraic group $G$ over $k$ and a
$G$-scheme $X$ with action $\sigma$, 
we have the action map
$\Psi := (\sigma,{\rm pr_X}): G \times X \to X \times X$.
We say that the action $\sigma: G \times X \to X$ is 
\emph{proper} if $\Psi$ is proper and
\emph{free} if $\Psi$ is a closed embedding.

Let $X$ be a scheme with $G$-action $\sigma$. 
Say that a morphism $\pi: X \to Q$ of $k$-schemes is a 
\emph{geometric quotient} of $X$ by $G$
if: $\pi \circ \sigma = \pi \circ {\rm pr}_X$,
$\pi$ is surjective, the image of $\Psi$ is $X \times_Q X$,
$U \subset Q$ is open if and only if $\pi^{-1} (U)$ is open,
and the structure sheaf $\mathcal{O}_Y$  is the subsheaf of
$\pi_* \left( \mathcal{O}_X \right)$ consisting of invariant
functions. 
We write $X\to X/G$ for the geometric quotient when it exists. 

If the geometric quotient $X/G$ exists, $X$ is called a 
\emph{principal $G$-bundle} over $X/G$ if $\pi: X \to X/G$ is
faithfully flat and $\Psi : G \times X \to X \times_{X/G} X$ is an
isomorphism.
By \cite[Lemme XIV 1.4]{ray-fais} this is equivalent to the condition
that $\pi$ is a locally isotrivial fibration with fiber $G$.
If $G$ acts freely on $X$ and the geometric quotient $X/G$ exists then
by \cite[Proposition 0.9]{mfk-git} it is a principle $G$-bundle.   

\subsubsection{}\label{ffd-section}

We frequently use faithfully flat descent for certain properties of
morphisms. We briefly summarize the main results we use, see
\cite[$\S$2]{EGA-IV-2} 
for details.  
Let {\bf P} be a property of morphisms of schemes that is stable under flat base change. 
Say that ${\bf P}$ satisfies
\emph{faithfully flat descent} if,
given $f : X \to Y$  and a faithfully flat morphism $Y' \to Y$,
such that $f': X \times_{Y} Y^{\prime} \to X$
satisfies {\bf P}, then $f$ satisfies {\bf P}.
Among the properties which satisfy descent are: separated, finite type,
proper, open immersion, closed immersion, finite, reduced, normal,
quasi-compact, regular, flat, \'etale and smooth. 

\subsubsection{}\label{sect-gmod}

Let $X$ be a $G$-scheme with action $\sigma$. 
A quasi-coherent
$\mathcal{O}_X$-module $\mathcal{F}$ is called 
a \emph{$G$-module} 
(see \cite[$\S$1.2]{tho-ekt}) if there is an isomorphism
$\phi: \sigma^* \mathcal{F} \to {\rm pr}_X^* \mathcal{F}$
of quasi-coherent $\mathcal{O}_{G \times X}$-modules such that the
co-cycle condition 
$\;{\rm p}^*_{23} (\phi) \circ \left( {\rm Id}_G \times \sigma
\right)^* (\phi) = \left( \mu \times {\rm Id}_X \right)^* (\phi) \;$
holds on $G \times G \times X$, where $p_{23} : G \times G \times X
\to G \times X$ is the projection onto the second and third factors, 
and $\mu : G \times G \to G$ is the multiplication on $G$. 
A morphism $f : M \to N$ between $G$-modules is called a 
\emph{$G$-morphism} if 
$\phi_N \circ \sigma^* (f) = {\rm pr}^*_X (f) \circ \phi_M$.

Say that $\mathcal{L}$ is \emph{$G$-linearizable} if it can be given a
$G$-module structure. 
A choice of such a structure is a \emph{$G$-linearization} of
$\mathcal{L}$.  
The set of $G$-linearized line bundles over $X$ is a group under
tensor product and is written ${\rm Pic}^G (X)$. 
An equivariant morphism $f : X \to Y$ induces the morphism
$f^* : {\rm Pic}^G(Y) \longrightarrow {\rm Pic}^G (X)$.
If $\pi : X \to Y$ is a principal $G$-bundle, then
$\pi^* : {\rm Pic} (Y) \to {\rm Pic}^G(X)$ is an isomorphism 
(see \cite[Ch. 1, $\S$ 3]{mfk-git} for details).

\subsubsection{}\label{gcat}

Projective and quasi-projective morphisms are not
stable under descent. 
To resolve the resulting difficulties in our situation we consider the
following category of $G$-schemes. 
Let $G-\mathbf{Sch}_k$  be the category whose objects are schemes $X$
with a $G$-action and which possess a $G$-linearizable ample line
bundle. 
Morphisms are equivariant maps. Similarly $G-\mathbf{Sm}_k$
consists of smooth $G$-schemes which possess a $G$-linearizable ample
line bundle.  

\begin{remark}
If $G$ is connected then every normal, quasi-projective $G$-scheme is
in $G-\mathbf{Sch}_k$ by \cite[Theorem 1.6]{sum-ec2}.
\end{remark}

A map $f:Y\to X$ in $G-\mathbf{Sch}_k$ is said to be an 
\emph{equivariant l.c.i.-morphism} provided that we can write 
$f= g\circ i$ where both $i$ and $g$ are in $G-\mathbf{Sch}_k$, 
$i:Y \to W$ is a regular closed embedding,  and  $g:W\to X$ is smooth
map. 

\begin{lemma}\label{des-pro}
\begin{enumerate}
\item Let $A$ be in $G-\mathbf{Sch}_k$ such that a principal $G$-bundle
  $A \to A/G$ exists with $A/G$ in $\mathbf{Sch}_k$. 
  Then a principal
  $G$-bundle $X \times A \to (X\times A)/G$ exists with
  $(X\times A)/G$ in 
  $\mathbf{Sch}_k$ for any $X$ in $G-\mathbf{Sch}_k$. 

\item 
  Suppose that principal $G$-bundles $X \to X/G$ and $Y \to Y/G$
  exist with $X/G$ and $Y/G$ in $\mathbf{Sch}_k$.
  Let $f: Y \to X$ be an equivariant projective (resp. an equivariant
  l.c.i.-morphism). 
  Then the induced morphism $\phi: Y/G \to X/G$  is projective
  (resp. $\phi$ is an l.c.i.-morphism).  
\end{enumerate}
\end{lemma}
\begin{proof}
\begin{enumerate}
\item Let $p_2 : X \times A \to A$ be the projection. 
  Then $X \times A$ has a $G$-linearizable $p_2$-ample line
  bundle because of the assumption on $X$. 
  The statement follows from \cite[Proposition 7.1]{mfk-git}.

\item 
  If $f$ is projective then it is proper and by descent $\phi$ is
  proper as well, but a proper, quasi-projective morphism between
  $k$-schemes is projective. 
  If $f$ is an equivariant l.c.i.-morphism we may factor it in 
  $G-\mathbf{Sch}_k$ as a regular immersion $i:Y \to W$ followed by a
  smooth morphism $g:W\to X$. 
  If $\mcal{L}$ is a linearizable ample bundle on $W$ then
  it is $g$-ample \cite[Proposition 4.6.13(v)]{EGA-II} and so by
  \cite[Proposition 7.1]{mfk-git} the principle $G$-bundle quotient
  $W \to W/G$ exists and $W/G\to X/G$ is smooth. 
  By descent the closed embedding $Y/G \to W/G$ is regular and
  therefore $\phi: Y/G \to X/G$ is an l.c.i.-morphism.  
\end{enumerate}
\end{proof}


\section{Equivariant Algebraic Cobordism}\label{EAC}

Fix a linear algebraic group $G$.
From now on, all the $G$-schemes to be considered are in the category
$G-\mathbf{Sch}_k$ introduced in the previous section. 
In this section we define the equivariant algebraic cobordism of a
$G$-scheme.

\subsection{Construction}
\begin{definition}
Say that $\{(V_i,U_i)\}$ is a \textit{good system of representations}
for a linear algebraic group $G$ if each $V_i$ a $G$-representation, 
$U_i\subseteq V_i$ is a $G$-invariant open satisfying the following
conditions: 
\begin{enumerate}
\item $G$ acts freely on $U_i$ and $U_i/G$ exists in $\mathbf{Sch}_k$.

\item For each $i$ there is a $G$-representation $W_{i}$ so that
  $V_{i+1} = V_i\oplus W_{i}$. 

\item $U_i \subseteq U_{i+1}$ and the inclusion factors as
  $U_i=U_{i}\oplus\{0\}\subseteq U_i\oplus W_{i} \subseteq U_{i+1}$.

\item $\lim_{i\to\infty}\dim V_{i} = \infty$. 

\item ${\rm codim}_{V_i} (V_i - U_i) < {\rm codim}_{V_j} (V_j - U_j)$,
  for $i < j$.
\end{enumerate}
\end{definition}

\begin{remark}
The existence of one such system follows from 
\cite[Remark 1.4]{tot-chow}.
\end{remark} 

Let $X$ be in $G-\mathbf{Sch}_k$ and consider a good system of
representations $\{ (V_i,U_i) \}$.
For convenience we write 
$X \times^G U_i = (X \times U_i)/G$. By Lemma \ref{des-pro} this
quotient exists, $X \times^G U_i$ is quasi-projective and the
morphisms $\phi_{ij}:X\times^{G} U_{i}\to X\times^{G}U_{j}$ are
l.c.i.-morphisms. If $X$ is smooth then by descent we have that 
$X \times^G U_i$ is in $\mathbf{Sm}_k$. 

\begin{definition}\label{constreac}
Let $G$ be a linear algebraic group.
Let $\{ (V_i,U_i)\}$ be a good system of representations.
For $X$ in $G-\mathbf{Sch}_k$ we define the 
\emph{equivariant algebraic cobordism group} as
\[
  \Omega_*^G (X) = 
  \varprojlim_{i}\Omega_* \left( X \times^G U_i \right). 
\]
The \emph{$n$-th equivariant algebraic cobordism group} 
of $X$ is defined as
\[
  \Omega_n^G (X) = \varprojlim_i \Omega_{n - \dim G + \dim U_i}
  \left( X \times^G U_i \right).
\]
If $X$ is an equidimensional and smooth $G$-scheme, define
$\, \Omega^*_G (X) = \varprojlim_{i}\Omega^* \left( X \times^G U_i
\right)\,$
and
$\Omega^n_G(X) = \varprojlim_{i} \Omega^{n} 
\left( X \times^G U_i \right)$.
Equivariant algebraic cobordism with rational coefficients is defined
by the completed tensor product and we write 
$\Omega_{G}^{*}(X)_{\Q} = \Omega_{G}^{*}(X)\widehat{\otimes}\Q$.  
\end{definition}

\begin{remark}
In the rest of the paper we consider the ungraded group $\Omega_{*}^{G}(X)$ for
simplicity. The reader interested in the graded situation may replace this by 
$\oplus\Omega_{n}^{G}(X)$, all results hold in this case with a few
appropriate changes.
\end{remark}

\begin{remark}\label{krishcompare}
In \cite{Krishna} another version of equivariant algebraic cobordism
was defined. The definition there yields a theory isomorphic to the
one we have defined, which we briefly explain. 
We restrict our discussion to smooth schemes so that we may index by
codimension but the discussion applies more generally provided one
indexes by dimension. 
Define the coniveau filtration 
\[
  F^{r} \Omega^{n}(X) 
= 
  \{ x \in \Omega^{n}(X) \, |\, j^{*}(x) = 0 \;\textrm{for some} \;
  j : X-S \to X, \, \,S\,\textrm{closed},\textrm{codim}(S)\geq r\}.
\] 
Then with our notational conventions for the meaning of the pairs
$\{(V_{j},U_{j})\}$, the definition of equivariant algebraic cobordism
given in \cite{Krishna} is
$\lim_{i}\Omega^{n}(X\times^{G}U_{i})/F^{c(i)}\Omega^{n}(X\times^{G}U_{i})$
where $c(i) = \textrm{codim}_{V_{i}}(V_{i}-U_{i})$.  
For a morphism $f:X\to Y$ we have $f^{*} (F^{r} \Omega^{n} (Y)) \subseteq
F^{r}\Omega^{n}(X)$. 
Also $F^{r}\Omega^{n}(X) = 0$ for $r>\dim X$. 
This means that whenever $j$ is large enough so that 
$c(j)>\dim(X\times^{G}U_{i})$ we have that 
$\phi^{*}_{ij}(F^{c(j)}\Omega^{n}(X\times^{G}U_{j})) = 0$. 

This implies that 
\[
  \Omega^{n}(X\times^{G}U_{j})/F^{c(j)}\Omega^{n}(X\times^{G}U_{j})
  \to
  \Omega^{n}(X\times^{G}U_{i})/F^{c(i)}\Omega^{n}(X\times^{G}U_{i})
\]
factors through 
$\Omega^{n}(X\times^{G}U_{i}) \to
\Omega^{n}(X\times^{G}U_{i})/F^{c(i)}\Omega^{n}(X\times^{G}U_{i})$ and
so the map of towers $\{\Omega^{n}(X\times^{G}U_{i})\}_{i}\to
\{\Omega^{n}(X\times^{G}U_{i})/F^{c(i)}\Omega^{n}(X\times^{G}U_{i})\}_{i}$
induces an isomorphism on inverse limits. 
\end{remark}

\subsubsection{}\label{well-def}

To see that our theory $\Omega_*^G$ is well-defined we will require
the following.

\begin{proposition}\label{prolim}
Let $\pi: E \to X$ a vector bundle over a scheme $X$ of rank $r$. 
Let $U \subseteq E$ be an open subscheme with closed complement
$S=E - U$.
\begin{enumerate}
\item If $X$ is affine and ${\rm codim}_{E}S > \dim X$ then
  $\pi_{|U}^*:\Omega_k(X) \to \Omega_{k + r }(U)$ is an
  isomorphism for all $k$. 

\item For a non affine scheme $X$, there is an integer $n(X)$ depending
  only on $X$, such that  
  $\pi^*_{|U} : \Omega_k (X) \to \Omega_{k + r}(U)$ is an 
  isomorphism for all $k$ whenever ${\rm codim}_E S > n(X)$.  
\end{enumerate}
\end{proposition}
\begin{proof}
\begin{enumerate}
\item We always have the commutative diagram
  \[
    \xymatrix{
              \Omega_{k+ r}(E) \ar[r]^{j^*} & 
              \Omega_{k+ r}(U) \ar[r] 
              & 0 \\
              \Omega_{k} (X) \ar[u]^{\pi^*}
                \ar[ur]_{\pi_{|U}^{*}}  
              & & 
             }
  \]
  where the top row is exact and $\pi^*$ is an isomorphism. 
  In particular 
  \begin{equation*}
    \pi_{|U}^*:\Omega_k (X) \to \Omega_{k + r} (U)
  \end{equation*}
  is surjective for all $k$.
  To show injectivity we proceed in cases. 
  First suppose that $E=\mathbb{A}^r \times X$ is trivial. 
  It suffices to find a section $s: X \to U$ of $\pi_{|U}$. 

  For each rational point $\xi \in \mathbb{A}^r$ define  
  $Z (\xi) = \left( \{\xi\}\times X \right) \cap S$.
  This is a closed subscheme $Z(\xi) \subseteq S$. If we can find a
  rational point $\xi \in \A^{r}$ such that $Z(\xi) = 
  \varnothing$ then the inclusion of $\{\xi\} \times X$ in $U$ defines 
  a section of $\pi_{|U}: U \to X$ and we are done. 
  Note that the condition $\mathrm{codim}_E \, S  > \dim X$ is
  equivalent to the condition $r > \dim S$. 

  Now consider the projection
  $\pi_{|S}:S \subseteq \mathbb{A}^{r}\times X \to \mathbb{A}^{r}$ and
  the closure of the image $\overline{\pi(S)} \subseteq
  \mathbb{A}^r$. 
  Since $S\to \overline{\pi(S)}$ is a dominant morphism between
  $k$-schemes we have $r > \dim S \geq \dim \overline{q(S)}$. 
  In particular the complement $\mathbb{A}^{r} - \overline{\pi(S)}
  \subseteq \mathbb{A}^{r}$ is a dense open subset.  
  We conclude that $\mathbb{A}^{r} - \overline{\pi(S)}$ has a rational
  point ($k$ is infinite). 
  Let $\alpha$ be a rational point in $\mathbb{A}^{r} -
  \overline{\pi(S)}$.
  Since $Z(\xi)\neq \varnothing$ whenever $\pi Z(\xi) = \xi$  and 
  $\alpha \notin \overline{\pi(S)}$ we must have that 
  $Z(\alpha) = \varnothing$.  

  More generally, since $X$ is affine every vector bundle
  admits a surjection from a trivial bundle.
  Let $p:\mathbb{A}^{N} \times X \to E$ be such a surjection of vector
  bundles on $X$. Let $W=p^{-1}U$ and set 
  $\tilde{S} = (\mathbb{A}^{N} \times X) - W$. 
  Then $\dim \tilde{S} = \dim S + N - r < N$. 
  By the proof of the first case above we have a section $X\to W$ and
  therefore we obtain a section $X\to W \to U$ of $\pi_{|U}: U \to X$. 

\item First assume that $X$ is affine. If ${\rm codim}_{E}S > \dim X$, 
  then by the first part of the proposition we have 
  $\pi^*_{|U} : \Omega_{k}(X) \to \Omega_{k+ {\rm rank} E}(U)$. 
  Thus in this case we may take $n(X) = \dim X$.

  For the general case we employ Jouanolou's trick to find an affine
  torsor $p: \tilde{X} \to X$ with $\tilde{X}$ affine. In this case we
  claim that we may take $n(X) = n(\tilde{X})$. 
  Consider the vector bundle $\tilde{E} := p^* E \to \tilde{X}$ and
  the open subscheme $\tilde{U}= p^* U \subseteq p^{-1}E$. 
  Set $\tilde{S} = \tilde{E} - \tilde{U}$. 
  From the first part of the proposition we see that since $\tilde{X}$
  is affine we have that    
  $\Omega_{*}(\tilde{X}) \to \Omega_{*}(\tilde{U})$ is an 
  isomorphism for 
  ${\rm codim}_{E}S = {\rm codim}_{\tilde{E}} \tilde{S} > \dim
  \tilde{X}$.  
  Since $\tilde{X} \to X$ and $\tilde{U} \to U$ are torsors for some
  vector bundle we have the following chain of isomorphisms 
  \[
    \xymatrix{
              \Omega_k (X) \ar[r]^{p^*} &
	      \Omega_k (\tilde{X}) \ar[r]^{\tilde{\pi}^*_{|U}} &
	      \Omega_{k + r} (\tilde{U}) &
	      \Omega_{k + r} (U) \ar[l]_{p^*_{|U}}.
             }
  \]
\end{enumerate}
\end{proof}

\begin{theorem}\label{eacwelldef}
For any $X \in G-\mathbf{Sch}_k$, $\Omega_*^G (X)$ is well defined. 
\end{theorem}
\begin{proof}
To see that our definition does not depend on the choice of the
sequence $\{ (V_i,U_i) \}$ we proceed as in \cite{tot-chow} by using
Bologomolov's double fibration argument. 
Let $\{ (V_i^{\prime},U_i^{\prime})\}$ be some other good system of
representations. Consider a fixed $U_i$. Since $G$ acts freely on
$U_i$ it acts freely on $U_i \oplus V_{j}^{\prime}$ too. 
Thus $X \times^G (U_i \oplus V_{j}^{\prime}) \to X \times^G U_i$ is a
vector bundle.  
The second part of Proposition \ref{prolim} says that there is an
integer $N_{i} = N(X\times^{G} U_{i})$ such that  
$\; \Omega_* (X \times^G U_i) \cong \Omega_*( X \times^G
(U_i \oplus U_j^{\prime}))\;$ for $j>N_{i}$. Thus 
$\ilim_{i}\Omega_* (X \times^G U_i) \cong
\ilim_{i}\ilim_{j}\Omega_*( X \times^G (U_i \oplus
U_j^{\prime}) )$. 
A  similar argument for $X \times^G U_i^{\prime}$ shows that 
$\ilim_{i}\Omega_* (X \times^G U_i^{\prime}) \cong
\ilim_{i}\ilim_{j}\Omega_*( X \times^G (U_i \oplus  
U_j^{\prime}) )$.
\end{proof}

\begin{example}\label{tact}
If $G \cong \langle e \rangle$ is the trivial linear algebraic group,
then the projections  
$X \times U_{i} \to U_{i}$ induce an isomorphism
$\Omega_* (X)\xrightarrow{\cong}\Omega_*^{\langle e \rangle} (X)$. 
Indeed, we have that $X \times^G U_i = X \times U_i$ for any $U_i$ in
the system.  
The statement follows from Proposition \ref{prolim}.
\end{example}

\subsubsection{}

We verify that the Mittag-Leffler condition holds on the
system defining $\Omega_G^*$.

\begin{lemma}\label{ML}
Let $\{(V_i,U_i)\}$ be a good system of representations. For every
$i<j$ in the system, let $\phi_{ij} : X \times^G U_i \to X \times^G
U_j$ be the induced morphism of schemes. Then 
\[
  \phi_{ij}^* : \Omega_* \left( X \times^G U_j \right) 
  \longrightarrow 
  \Omega_* \left( X \times^G U_i \right)
\]
is a surjection. 
\end{lemma}
\begin{proof}
The morphism $U_i\to U_j$ factors as 
$U_i \to U_i \oplus W  \subseteq U_j$, where $W$ is a representation 
(depending on $i$ and $j$). 
Now, $X \times^G U_i \to X \times^G (U_i\oplus W)$ is
the inclusion of the zero section of a vector bundle and $ X \times^G
(U_i \oplus W) \subseteq X \times^G U_j$ is an open inclusion. 
Both maps induce surjections on algebraic cobordism. 
\end{proof}

\subsubsection{}\label{sect-eclass}

Let $\{ (V_i,U_i) \}$ be a good system of representations. 
Let $X$ be in $G-\mathbf{Sch}_k$.
For $U_i \subseteq U_j \,$, write
$\phi_{ij} : X \times^G U_i \to X \times^G U_j$ for the induced morphism.
By definition, a $G$-equivariant cobordism class on $X$
is a sequence of cobordism classes 
$\{ \alpha_i \in \Omega_* (X \times^G U_i) \mid 
\phi_{ij}^* ([\alpha_j]) = [\alpha_i] \}$.

An equivariant projective morphism $f:Y\to X$ with $Y$ in
$G-\mathbf{Sm}_k$ defines an equivariant cobordism class as follows. 
The assumption on $Y$ implies that $f_{i}:Y \times^G U_i \to X
\times^G U_i$ is projective and $Y \times^G U_i$ is smooth. 
Therefore for each $i$ we have the induced cobordism class
$[f_{i}:Y \times^G U_i \to X \times^G U_i]$ in 
$\Omega_* (X \times^G U_i)$. 
By descent we have the transverse Cartesian square
\[
  \xymatrix{
            Y \times^{G} U_{i} \ar[r] \ar[d] & 
            Y \times^{G} U_{j} \ar[d] \\
            X \times^{G} U_{i} \ar[r]^{\phi_{ij}} & 
            X \times^{G} U_{j}
           }
\]
which shows that 
$\phi_{ij}^{*}[f_{i}:Y \times^G U_i \to X \times^G U_i] 
= [f_{j}:Y \times^G U_j \to X \times^G U_j]$. 
Therefore $f:Y\to X$ induces the class 
$([Y \times^G U_i \to X \times^G U_i])$ in $\Omega_*^G (X)$. 
Note however that $Y\to X$ is not a unique representation for this
class.

\subsection{Computations}\label{computations}

\subsubsection{Coefficient Ring in the Case of a Torus
  Action}\label{CRTA} 

Let $T = \left( \mathbb{G}_m \right)^r$.
Consider the good system of representations 
$\{ (V_i, U_i) \}$, where $V_i = \left( \mathbb{A}^i \right)^r$ 
and $U_i = \left( \mathbb{A}^i - \{ 0 \} \right)^r$.
The action of $T$ on $V_{i}$ is given by letting the $k$-th factor of 
$\mathbb{G}_{m}$ act on the $k$-th factor of $\mathbb{A}^{i}$ via the
formula 
\[
  \mathbb{G}_m \times \mathbb{A}^i \to \mathbb{A}^i,
\qquad
  (g, a_1, \ldots, a_i) \mapsto (g \cdot a_1, \ldots, g \cdot a_i).
\]
Then $T$ acts freely on $U_i$ and 
$U_i / T = \left( \mathbb{P}^{i-1} \right)^r$. 
A direct computation shows
\[
  \Omega^*_T (k) 
=
  \varprojlim_i \Omega^* \left( \left( \mathbb{P}^{i-1} \right)^r
  \right)
=
  \varprojlim_i \frac{\Omega^*(k) [t_1, \ldots, t_r]}
  {\left( t_1^{i-1}, \ldots, t_r^{i-1} \right)}
=
  \Omega^*(k) [[ t_1, \ldots, t_r ]],
\]
where each $t_i$ is a variable of degree 1.

\subsubsection{$\mathbb{P}^n$ with
  a Weighted $\mathbb{G}_m$-Action} 

Let $\mathbb{G}_m$ act on $\mathbb{P}^{n}$ with the action 
\[
  \mathbb{G}_m \times \mathbb{P}^{n} \to \mathbb{P}^{n}, 
\qquad
  \left( g , (a_0 : \cdots : a_n) \right) \mapsto 
  \left( g^{m_0} \cdot a_0 : \cdots : g^{m_n} \cdot a_n \right).
\]
Let $\{ (V_i,U_i) \}$ be the good system of representations considered
in $\S$\ref{CRTA}. For each $U_i$ in the system we have the 
$\mathbb{P}^n$-bundle
$\mathbb{P}^{n} \times^{\mathbb{G}_m} U_i \to \mathbb{P}^{i-1}$ is
a $\mathbb{P}^n$-bundle. 
As a $\mathbb{P}^{i-1}$-scheme, we have that 
$\mathbb{P}^{n} \times^{\mathbb{G}_m} U_i
\cong 
\mathbb{P} \left( \mathcal{O} (m_0) \oplus \cdots \oplus \mathcal{O}
(m_n) \right).$
By \cite[Lemma 4.1.4]{lev&mor-book} we get
\[
  \Omega^* \left( \mathbb{P}^{n} \times^{\mathbb{G}_m} U_i \right)
=
  \frac{\Omega^* \left( \mathbb{P}^{i-1} \right) [\xi]}
  {\left( \xi - \mathrm{c}_1 \left( \mathcal{O} (m_0) \right) \right) 
  \cdots  
  \left( \xi - \mathrm{c}_1 \left( \mathcal{O} (m_n) \right) \right)},
\]
where $\xi$ is a variable of order one. If we let 
$t=\mathrm{c}_{1}(\mcal{O}(1))\in \Omega^{*}(\P^{i-1})$ then 
$\mathrm{c}_1 \left( \mathcal{O} (a) \right) = [a]_{\Omega}(t)$ where
$[a]_{\Omega}(t)$ 
is defined by $F_{\Omega}(t,[a-1]_{\Omega}t)$. 
Taking the limit we obtain
\[
  \Omega^*_{\mathbb{G}_{m}} \left( \mathbb{P}^{n} \right) 
= \frac{\Omega^* (k) [[ t ]] [ \xi ]}{\left( \xi - [m_0]_{\Omega}
  (t )\right) \cdots \left( \xi - [m_n]_{\Omega} ( t) \right)}.
\]

\subsubsection{A Torus Acting Trivially}

Let $T= \mathbb{G}_{m}^{\times r}$ and let $X$ be a scheme considered
in $T-\mathbf{Sm}_k$ with the trivial action.
Then 
$\Omega_T^{*} (X) \cong  \Omega^* (X) \widehat{\otimes}_{\Omega^*(k)} 
\Omega^*_T (k)$ of $\Omega^*(k)$-modules.
This follows from 
\[
  \Omega^*_T (X) 
=  
  \varprojlim_i \Omega^* \left( X \times \left( \mathbb{P}^{i-1}
  \right)^r \right) \\
\cong
  \varprojlim_i \frac{\Omega^* (X) [t_1, \ldots , t_r]}
  {\left( t_1^i, \ldots, t_r^i \right)}  
\cong
  \Omega^* (X) [[ t_1, \ldots, t_r ]],
\]
where the first isomorphism is the statement of 
\cite[Lemma 4.1.4]{lev&mor-book}. 

\subsubsection{Coefficient Ring in the Case of a $GL_n$-Action}

Let $M_{n \times (n+i)}$ be the space of $n \times (n+i)$ matrices.
Consider the good system of representations $\{ (V_i, U_i ) \}$, where
$V_i = M_{n \times (n+i)}$ with $\mathrm{GL}_n$ acting by
multiplication on the left and $U_i$ is the subset of matrices of 
maximal rank. 
We have that $U_i / \mathrm{GL}_n \cong {\rm \mathrm{Gr}} (n,n+i)$, 
where $\mathrm{Gr}(n,n+i)$ is the Grassmannian of $n$-planes in
$k^{n+i}$.   
Let $\mathbb{F}_{n+i}$ denote the variety of complete flags in
$k^{n+i}$.  
Let $\phi:\mathbb{F}_{n+i}\to \mathrm{Gr}(n,n+i)$ be the map which
sends the flag 
$Fl_{n+i} = \{0\subseteq F^{1}\subseteq\cdots \subseteq
F^{n+i}=k^{n+i}\}$ to the $n$-plane $F^{n}\subseteq k^{n+i}$.  
The induced map $\phi^{*}$ is injective on cobordism with rational
coefficients by Proposition \ref{lhrat} and because the Grassmannian
is cellular we conclude that
$\phi^{*}:\Omega^{*}(\mathrm{Gr}(n,n+i))\to\Omega^{*}(\mathbb{F}_{n+i})$
is injective integrally.
Let $\mcal{W}_{k}$ be the tautological $k$-plane bundle on
$\mathbb{F}_{n+i}$ (i.e. the fiber of $\mcal{W}_{k}$ on the flag
$Fl_{n+i}$ is $F^{k}$) and let $\mcal{L}_{k} = \mcal{W}_{k}/\mcal{W}_{k-1}$. 
In \cite[Theorem 2.6]{jens&val-schubert} the cobordism of the complete
flag variety is shown to be 
\[
  \Omega^*(\mathbb{F}_{n+i}) 
\iso 
  \Omega^* (k) [x_1, \ldots x_{n+i}]/IS_{n+i},   
\]
where $x_j=c_{1}(\mcal{L}_{j})$ and $S_{n+i}$ is the graded ring of
symmetric polynomials in the $x_j$ with coefficients in
$\Omega^{*}(k)$ and $IS_{n+i}$ is the ideal generated by 
the symmetric polynomials of strictly positive polynomial degree.

The cobordism of $\mathrm{Gr}(n,n+i)$ is generated by the Chern
classes $c_{r}(\mcal{E}_{n})$, where $\mcal{E}_{n}$ is the
tautological 
$n$-plane bundle on $\mathrm{Gr}(n,n+i)$ 
(so $\phi^{*}\mcal{E}_{n} = \mcal{W}_{n}$). 
The total Chern class of $\mcal{W}_{n}$ is 
$c(\mcal{W}_{n}) = \prod_{k=1}^{n} c(\mcal{L}_{k})
= \prod_{k=1}^{n}(1+x_{k})$ from which we see that
$\phi^{*}c_{k}(\mcal{E}_{n})$ is the $k$-th elementary symmetric
polynomial in the $x_{1},\ldots, x_{n}$. Thus $\phi^{*}$ gives us
the identification 
\[
  \Omega^{*}(\mathrm{Gr}(n,n+i)) 
= 
  S_{n}/IS_{n+i} \subseteq \Omega^* (k) [x_1, \ldots x_{n+i}]/ IS_{n+i},
\]
and therefore we see that 
$\Omega^*_{\mathrm{GL}_n} (k) 
\iso 
\Omega^{*}(k) [[\eta_1,\ldots, \eta_n]]$,
where $\eta_j$ is of degree $j$.

\subsubsection{Roots of unity}

Let $\mu_n$ be the algebraic group of roots of unity.
Let $X$ be in $\mu_n - \mathbf{Sm}_k$.
We show that $\Omega^*_{\mu_n} (k) \cong \Omega^*
(k)[[\xi]] / [n]_{\Omega} \cdot \xi$.

Consider the Kummer sequence
$1\to \mu_n \to \G_m \xrightarrow{(-)^n} \G_m \to 1$,
which is exact in the \'etale topology. 
We obtain an \'etale $\G_m$-torsor 
\[
 \mathbb{A}^i - \{0\} / \mu_n
\longrightarrow 
  \mathbb{A}^i - \{0\} / \mathbb{G}_m = \mathbb{P}^{i-1}. 
\]
Since ${\rm Pic} (X) = {\rm H}^1_{Zar}(X; \mathcal{O}_{X}^{*}) = 
{\rm H}^{1}_{et}(X; \mathbb{G}_m)$, we know that
$\mathbb{G}_m$-torsors correspond to line bundles on $X$.
The line bundle associated to this $\G_m$-torsor is
$\mathcal{O}_{\mathbb{P}^{i-1}}(-n) = L$.
Let $\pi : L \to \mathbb{P}^{i-1}$ denote the structural morphism, and  
write $L_{0}$ for the complement of the zero section $s$. 
There is an embedding $j: \A^{i} - \{0\} / \mu_n \to L$ such that 
$j \left( \A^{i} - \{0\} / \mu_n \right) = L_{0}$. 
Now from the localization sequence
\[
  \Omega^{*} (\mathbb{P}^{i-1})
\xrightarrow{s_*} 
  \Omega^* (L) 
\to 
  \Omega^* (L_{0}) 
\to 
  0
\]
we see that $\Omega^* (\mathbb{P}^{i-1}) \xrightarrow{\pi^*} \Omega^*
(L) \to \Omega^* (L_{0})$ is surjective.  

Since 
$s^*:\Omega^* (L) \to \Omega^* (\mathbb{P}^{i-1})$ 
is an isomorphism and $s^*s_* (L) = \mathrm{c}_1(L)$ 
\cite[Proposition 4.1.15]{lev&mor-book}
it follows that 
$\Omega^*(\mathbb{A}^{i} - \{0\} / \mu_n) 
= \Omega^*(\mathbb{P}^{i-1}) / \mathrm{c}_{1}(\mathcal{O}(-n))$.
The result follows from the equalities 
$\mcal{O}(-n) = \mcal{O}(-1)^{\otimes n}$ and $\xi = c_1\mcal{O}(-1)$.


\section{Formal Properties of Equivariant Algebraic Cobordism}
\label{secfour}


In this section we establish some properties of $\Omega_*^G$.
Mainly the properties are of two types, one the equivariant analogues
of the formal properties of 
an oriented Borel-Moore theory, and the other are expected from an
equivariant cohomology theory.
Proceeding as in the end of $\S$ \ref{well-def}, shows that
all of the following properties are independent of the choice of a
good system of representations.

From now on, let $\{ (V_i,U_i) \}$ be a fixed good system of
representations for a linear algebraic group $G$.  
All the $G$-schemes are in the category
$G-\mathbf{Sch}_k$ introduced in $\S$ \ref{gcat}. 

\subsection{Variances}\label{sect-var}

Since any l.c.i. morphism $f: Y \to X$ of relative dimension $d$ in 
$G-\mathbf{Sch}_k$ induces for any $U_i$ in the system a
l.c.i. morphism $f_i : X \times^G U_i \to Y \times^G U_i$ 
in {\bf Sch}$_k$, we obtain a sequence of pull-backs maps 
$f_i^* : \Omega_* (X \times^G U_i) \to \Omega_* (Y \times^G U_i)$. 
By naturality of $\Omega_*$ we have a functorial
\emph{pull-back map}
\[
   f^*_G := \varprojlim_i f_i^* : \Omega_*^G (X) \to
\Omega_{*+d}^G (Y),
\]
which is a morphism of Abelian groups.

If $f:Y \to X$ is a projective morphism in 
$G-\mathbf{Sch}_k$, 
by Lemma \ref{des-pro} we have a sequence of projective morphisms 
$f_i : Y \times^G U_i \to X \times^G U_i$ in {\bf Sch}$_k$. 
We have a transverse Cartesian diagram 
\[
  \xymatrix{
            Y \times^G U_i \ar[r]^{f_i} \ar[d] &
            X \times^G U_i \ar[d] \\
            Y \times^G U_j \ar[r]_{f_j} &
            X \times^G U_j 
           }
\]
for any $i<j$. 
Therefore the $f_{i}$ are compatible with the transition maps in the
system and we obtain an induced functorial 
\emph{push-forward map}
$\, f^G_* :\Omega_*^G (Y) \to \Omega_*^G (X)$.

\begin{proposition}\label{epf}
\begin{enumerate}
\item Let $f: X^{\prime} \to X$ and $g: Y^{\prime} \to Y$ be
  l.c.i. morphisms in $G-\mathbf{Sch}_k$.
  If $u \in \Omega_*^G (X)$ and $v \in \Omega_*^G (Y)$ then
  $\left( f \times g \right)^*_G (u \times v) 
  = f^*_G (u) \times g^*_G (v)$. 

\item Let $f: X^{\prime} \to X$ and $g: Y^{\prime} \to Y$ be
  projective morphisms in $G-\mathbf{Sch}_k$.
  If $u^{\prime} \in \Omega_*^G (X^{\prime})$ and 
  $v^{\prime} \in \Omega_*^G (Y^{\prime})$ then
  $\left( f \times g \right)_*^G (u^{\prime} \times v^{\prime}) 
  = f_*^G (u^{\prime}) \times g_*^G (v^{\prime})$.

\item Let $f: X \to Z$ be a projective morphism in 
  $G-\mathbf{Sch}_k$, and $g : Y \to Z$ is a l.c.i. morphism in  
  $G-\mathbf{Sch}_k$.
  If we have a transverse Cartesian diagram 
\[
  \xymatrix{
            W \ar[r]^{g^{\prime}} \ar[d]_{f^{\prime}} &
            Y \ar[d]^f \\
            Z \ar[r]_{g} &
            X 
           }
\]
  then
  $g^*_G \circ f_*^G = f^{\prime \, G}_* \circ g^{\prime \, *}_G$.
\end{enumerate}
\end{proposition}
\begin{proof}
(1) and (2) follows from taking the limit of the corresponding
identities.
Now we proceed to show (3).
By descent and flat base change for ${\rm Tor}_n$
we obtain the transverse Cartesian diagram
\[
  \xymatrix{
            W \times^G U_i \ar[r]^{g_i^{\prime}} 
                           \ar[d]^{f_i^{\prime}} &
            Y \times^G U_i \ar[d]_{f_i} \\
            Z \times^G U_i \ar[r]_{g_i} &
            X \times^G U_i
           }
\]
in {\bf Sch}$_k$ for each $i$. Since $g_{i}^{*} \circ f_{i\,*} 
= f^{\prime }_{i\,*} \circ g_{i}^{\prime \, *}$ 
and this is compatible with the
transition maps in the inverse system, we obtain 
$ g^*_G \circ f_*^G = f^{\prime \, G}_* \circ g^{\prime \, *}_G$.
\end{proof}

\subsection{Localization Sequence}\label{sect-ls}

\begin{theorem}
Let $X$ in $G-\mathbf{Sch}_k$.
Let $\imath : Z \to X$ be an invariant closed subscheme and let $j: U
\to X$ be the open complement. 
Then we have the exact sequence  
\begin{equation*}\label{LS}
  \xymatrix{
            \Omega_*^G (Z)
              \ar[r]^{\imath_*^G} & 
            \Omega_*^G (X)  \ar[r]^{j^*_G} &
            \Omega^G_* (U) \ar[r]  &
            0 
           }.
  \tag{{\bf LS}}
\end{equation*}
\end{theorem}
\begin{proof}
The proof follows directly from (\ref{AC}), Lemma \ref{ML} and
naturality of the push-forwards and pull-backs.
\end{proof}

\subsection{Projective Bundle Axiom}\label{sect-pb}

Let $E \to X$ be a $G$-equivariant vector bundle of rank $r+1$.
We know that $E \times^G U_i$ is a vector bundle of
rank $r+1$ over $X \times^G U_i$ for each $U_i$ in the system. 
Let $q_i : \mathbb{P}_i:= \mathbb{P}\left( E \times^G U_i \right) \to
X \times^G U_i$ be the associated projective bundle. 
As a projective bundles over $X \times^G U_i$ we have that 
$\mathbb{P} ( E \times^G U_i )$ is isomorphic to 
$\mathbb{P} (E) \times^G U_i$. 
For convenience we will work with $\mathbb{P} ( E \times^G U_i )$.
By (PB) for $\Omega_*$ we have an isomorphism
\begin{equation*}
  \Phi_i := 
  {\textstyle \prod_{n=0}^{r} \; \xi_i^n \cdot q_i^*} : 
  \prod_{n=0}^{r} \Omega_{* -r + n} \left( X \times^G U_i
  \right) 
  \longrightarrow
  \Omega_* \left( \mathbb{P}_i \right),  
\end{equation*}
where $\xi_i = {\rm c}_1 \left( \mathcal{O}_{\mathbb{P}_i} (1)
\right)$.  
Consider the morphisms
$\phi_{ij} : E \times^G U_i \to E \times^G U_j$
for $i<j$.
Notice that 
$\, \mathcal{O}_{\mathbb{P}_i} (1) =
\mathbb{P}(\phi_{ij})^* \left( \mathcal{O}_{\mathbb{P}_j} (1)
\right)$, so $\mathbb{P}(\phi_{ij})^* \left( \xi_j \right) = \xi_i$. 
Thus, the isomorphisms $\Phi_{i}$ induce the isomorphism
\[
  \varprojlim_{i}{\textstyle 
  (\prod_{n=0}^{r} \; \xi_i^n \cdot q_i^*)}:
  \varprojlim_i \left(\prod_{n=0}^{r}  \Omega_{*-r+n} 
  \left( X \times^G U_i \right) \right)
  \longrightarrow
  \varprojlim_i \Omega_* \left( \mathbb{P} (E \times^G U_j) \right).
\]
We have proved the following.

\begin{proposition}\label{G-PB}
Let $E \to X$ be a $G$-equivariant vector bundle of rank $r+1$ in
$G-\mathbf{Sch}_k$. 
With the notation above, set $\xi_*^G := \varprojlim_i \xi_i$, 
and $\Omega^*_G \left( \mathbb{P} (E) \right) := \varprojlim_i
\Omega_* \left( \mathbb{P} (E \times^G U_i) \right)$.
Then
\begin{equation}\label{EPB}
  \Phi^G_{X,E} :=
  \prod_{n=0}^r \left( \xi_*^{G} \right)^n \cdot q^*_G :
  \prod_{n=0}^{r} \Omega^G_{*-r+n} (X)  
  \longrightarrow
  \Omega_*^G \left( \mathbb{P} (E) \right) 
\tag{{\bf PB}}
\end{equation}
is an isomorphism.
\end{proposition}

\subsection{Extended Homotopy Axiom}\label{sect-eh}

\begin{proposition}
Let $E \to X$ be a $G$-vector bundle of rank $n$ and $p: Y \to X$ be a
$E$-torsor in $G-\mathbf{Sch}_k$, then
\begin{equation*}
  p^*_G : \Omega_*^G (X) \longrightarrow \Omega_{*+n}^G (Y) 
\tag{{\bf EH}}
\end{equation*}
is an isomorphism.  
\end{proposition}
\begin{proof}
We have that
$E \times^G U_i \; \times_{X \times^G U_i} \; 
Y \times^G U_i$ is isomorphic to 
$\left( E \times U_i \; \times_{X \times U_i} 
Y \times U_i \right) / G$ for each $i$.
Therefore the action map $\Psi:E\times_{X}Y \to Y\times_{X}Y$ induces
an action map 
\[
  \Psi_i : E \times^G U_i \; \times_{X \times^G U_i} \; Y \times^G U_i  
\longrightarrow
  Y \times^G U_i \; \times_{X \times^G U_i} \; Y \times^G U_i
\]
and $\Psi_i$ is an isomorphism for
all $i$. 
This makes $Y \times^G U_i$ an $E \times^G U_i$-torsor over 
$X \times^G U_i$. The result then follows from the extended homotopy
property for algebraic cobordism. 
\end{proof}

\begin{corollary}[Homotopy Invariance]
Let $X$ be in $G-\mathbf{Sch}_k$. 
If $G$ acts linearly on $\mathbb{A}^n$ then  
$p^*_G : \Omega_*^G (X) \to
\Omega_{*+n}^G \left( \mathbb{A}^n \times X \right)$
is an isomorphism.
\end{corollary}

\subsection{Chern Classes of $G$-Equivariant Vector Bundles}

Let $E \to X$ be an $G$-equivariant vector bundle in $G-\mathbf{Sch}_k$.
We define the $n$-th \emph{$G$-equivariant Chern operator} 
$\tilde{{\rm c}}_n^G (E)$ as
\[
  \tilde{{\rm c}}^G_n (E) := \varprojlim_i \tilde{{\rm c}}_{n,i}(E),
\]
where $\tilde{{\rm c}}_{n,i} (E)$ is the n-th Chern operator of 
$E \times^G U_i \to X \times^G U_i$ induced by (PB).
If $X$ is smooth, given a $G$-equivariant vector bundle $E$ over $X$,
define the $n$-th \emph{$G$-equivariant Chern class} $\mathrm{c}_n^G
(E)$ in $\Omega_{-n}^G (X)$ as 
$\mathrm{c}_n^G (E):= \tilde{\mathrm{c}}_n^G (E) (1_X)$.
\begin{remark}
Restrict to smooth schemes. 
\begin{enumerate}
\item We could have defined $\mathrm{c}_n^G (E)$ by means of 
  ({\bf PB}) following Grothendieck's method \cite{gro-cla}, so that 
  $\sum_{n=0}^{r} (-1)^n {\rm c}_n^G (E)  \, \xi_G^{r-n} = 0$ 
  holds, where $r = \mathrm{rank} (E)$.

\item If $E = \oplus_{i=1}^r L_i$, with $L_i \to X$ a $G$-linear
  bundle, then ${\rm c}_n^G (E)$ is the $n$-th elementary symmetric
  polynomial at the ${\rm c}_1^G (L_i)$. 
\end{enumerate}
\end{remark}
Our equivariant Chern operators have the expected properties.
\begin{lemma}
Let $X$ and $Y$ be in $G-\mathbf{Sch}_k$. 
\begin{enumerate}
\item{\rm (Commutativity)} 
  Let $E$ and $F$ be $G$-vector bundles over $X$. 
  For any $i$ and $j$ we have that 
  $\, \tilde{{\rm c}}_i^G (E) \circ \tilde{{\rm c}}_j^G (F)
= \tilde{{\rm c}}_j^G (F) \circ \tilde{{\rm c}}_i^G (E)$.

\item {\rm (Naturality)} 
  For any l.c.i. morphism $f : Y \to X$ in $G-\mathbf{Sch}_k$
  and any $G$-equivariant vector bundle $E \to X$, we have that 
  $\,\tilde{{\rm c}}_n^G \circ f^* E = f^*_G \circ \tilde{{\rm c}}_n^G
  (E)$, for all $n \geq 0$.

\item {\rm (Whitney Formula)} 
  If $0 \to E^{\prime} \to E \to E^{\prime  \prime} \to 0$ is an
  $G$-equivariant exact sequence of $G$-vector bundles over $X$, then  
  $\tilde{{\rm c}}_n^G (E) 
= \sum_{l=0}^n \tilde{{\rm c}}_l^G \left( E^{\prime} \right) 
  \tilde{{\rm c}}_{n-l}^G \left( E^{\prime \prime} \right)$,
  for all $n \geq 0$.
\end{enumerate}
\end{lemma}

\subsection{Restriction Maps}

In this section we relate $\Omega_*^G$ and $\Omega_*$ via restriction
maps, which can be defined via $\langle e \rangle \to G$ or by
restricting to the fiber. 
We show that both agree up to isomorphism on $\Omega_*$. 

\subsubsection{Restricting the Action}\label{rest-subgp}

Let $H \subseteq G$ be a closed normal subgroup of $G$. 
Since the induced action of $H$ on each $U_i$ is free and the quotient
$U_{i}/H = G/H\times^{G} U_{i}$ exists we see that $\{(V_{i},U_{i})\}$
is also a good system for $H$.  
Moreover, we have smooth morphisms $X \times^H U_i \to X \times^G
U_i$ that induces
\[
  {\rm res}_{G,H} : \Omega_*^G (X) \longrightarrow \Omega_*^H (X).
\]
When $H = \langle e \rangle$, we will use ${\rm res}_G$ to denote 
${\rm res}_{G,\langle e \rangle}$. 
From Example \ref{tact} we have the natural isomorphism 
$\Omega_* (X)\cong\Omega_*^{\langle e \rangle} (X)$ and so we obtain
${\rm res}_G : \Omega_*^G (X) \longrightarrow \Omega_* (X)$.

\subsubsection{Restriction to the Fiber}\label{rest-fib}

Assume $X \in G-\mathbf{Sch}_k$ to be irreducible. 
Let $\eta \in U/G$ be a rational point, where $U$ is the initial
$G$-invariant open in the system being considered.
For each $i$, the projection $X \times U_i \to U_i$ induces a flat
morphism $X \times^G U_i \to U_i / G$ whose fiber over a rational
point $\eta_i$ of $U_i / G$ equals $X$, where $\eta_i$ is the image of
$\eta$ under the canonical morphism $U/G \to U_i/G$.  
We have an induced morphism 
$\, {\rm res}^G_i (\eta): \Omega_* (X \times^G U_i) \to \Omega_* (X)$.
For any $i < j$ in the system 
we have an induced commutative diagram
\[
  \xymatrix{
           \Omega_*^G (X) \ar[d] \ar[dr] & & \\
	   \Omega_* (X \times^G U_j) \ar[r] \ar[rd] &
	   \Omega_* (X \times^G U_i) \ar[d] \\
	   & \Omega_*(X).
	  }
\]
Thus, given a rational point $\eta$ in $U/G$ and for any $i$ we have
the \emph{restriction map} 
\[
  {\rm res}_{\Omega}^G (\eta) : \Omega_*^G (X) 
\longrightarrow 
\Omega_* \left( X \times^G U_i \right) 
\stackrel{{\rm res}_i^G (\eta)}{\longrightarrow} 
\Omega_* (X).
\] 
If the group $G$ is clear from the context, we will use the notation
${\rm res}_{\Omega} (\eta)$.
When $G = \langle e \rangle$, the restriction 
${\rm res}^{\langle e \rangle}_{\Omega} (\eta)$ is 
precisely the isomorphism of Example \ref{tact}.

\subsubsection{Comparison of Restrictions}

Fix a rational point $\eta \in U/G$ as in the previous section.
Let $H$ be a normal closed subgroup of $G$. 
For every $U_i$ in the system, let $P_H(i)$ and $P_G(i)$ be points in
$U_i/H$ and $U_i/G$ respectively, so that $P_H(i) \mapsto P_G(i)$
under the canonical map $U_i/H \to U_i/G$, where $\eta \mapsto P_G(i)$
under $U/G \to U_i / G$. 
We have the commutative diagram 
\[
  \xymatrix{
            &&
            X \ar[r] \ar@{=}[dll] \ar[d] &
	    X \times^H U_i \ar[d] \ar[dll] \\
	    X \ar[r] \ar[d] &
	    X \times^G U_i \ar[d] &
	    P_H(i) \ar[dll] \ar[r] &
	    U_i/H \ar[lld] \\
	    P_G(i) \ar[r] &
	    U_i/G &
	    &
	   }
\]
with Cartesian faces induced by the fiber squares.

Hence we have a commutative diagram
\[
  \xymatrix{
            \Omega_*^G (X) \ar[r]^{{\rm res}_{G,H}} 
              \ar[d]_{{\rm res}_{\Omega}^G (\eta)} &
	    \Omega_*^H (X) \ar[d]^{{\rm res}_{\Omega}^H (\eta)} \\
	    \Omega_* (X) \ar[r]_{\mathrm{Id}} &
	    \Omega_* (X).
	   }
\]
For $H = \langle e \rangle$, we have seen that
${\rm res}^H_{\Omega} (\eta)$ is an isomorphism, so 
${\rm res}^G_{\Omega} (\eta)$ and ${\rm res}_G$ are equal up to
isomorphism on $\Omega_*(X)$. 
In particular, ${\rm res}^G_{\Omega} (\eta)$ is independent of
$\eta$. From now on, we will denote by ${\rm res}^G_{\Omega}$ the 
restriction to the fiber map.
With this notation, we have proved 
$\;{\rm res}_{\Omega}^G = {\rm res}_{\Omega}^{\langle e \rangle} 
\circ {\rm res}_G$.

\subsubsection{}

We have the following. 

\begin{theorem}[Induction]\label{ind}
Let $H$ be a closed normal subgroup of $G$. 
Consider $G$ as an $H$-scheme with the action $(h , g) \mapsto g
h^{-1}$. 
Let $X$ be in $G-\mathbf{Sch}_k$.  
Then 
\[
  \Omega_*^H (X) \cong \Omega_*^G (X \times^H G),
\]
where $X\times^{H} G$ is given a $G$-action via its action on $G$.
Moreover, if $X$ is obtained by restriction of a $G$-action then 
\[
  \Omega_*^H (X)\cong \Omega_*^G (X \times G/H).
\]
These isomorphisms are natural with respect to the
variances. 
\end{theorem}
\begin{proof}
Let $\{V_{i}, U_{i}\}$ be a good system of representations for $G$. By
restricting the action this provides a good system of representations
for $H$ as well. The first statement follows from the isomorphism 
$(X\times^{H} G)\times^{G} U_{i} \to X\times^{H} U_{i}$,
given by $([x,g], u) \mapsto (x, g^{-1}u)$.

If $X$ is an $H$-scheme obtained by restricting a $G$-action then 
$X \times^H G \to X \times G/H$, given by $[x,a] \mapsto (ax, aH)$, is
an isomorphism of $G$-schemes.
The second statement now follows from the first. 
\end{proof}

\subsection{Geometric Quotients}

In this section we assume the geometric quotient $p : X \to X/G$
exists.  
We compare the ordinary cobordism of $X/G$ and the equivariant
algebraic cobordism $X$. 
As a consequence we see that the ordinary cobordism with rational
coefficients of $X/G$, for smooth $X$, is equipped with a natural
ring structure.

The fiber of $\pi_i : X \times^G U_i \to X/G$ over $x \in X/G$ is given
by $\pi_{i}^{-1}(x) = U_i/G_x$. 
Since $U_i/G_x$ is smooth, by \cite[Corollary 6.3.24]{Liu:AG} we have
that each $\pi_i$ is an l.c.i. morphism. 
The morphisms 
$\pi_i^* : \Omega_k (X/G) \to \Omega_{k + \dim U_i} (X \times^G U_i)$
induce
\[
  \pi^* : \Omega_* (X/G) \longrightarrow \Omega^G_{* + \dim G} (X).
\]

\begin{proposition}\label{trivstab}
Let $X$ be in $G-\mathbf{Sch}_k$. 
Assume that $X \to X/G$ is a principal $G$-bundle.
Then $\pi^* : \Omega_*(X/G) \to \Omega^G_{*+\dim G} (X)$ is an
isomorphism.
\end{proposition}
\begin{proof}
Since the stabilizers are trivial, $X \times^G V_i \to X/G$ is a vector
bundle for every representation $V_i$.  
By Proposition \ref{prolim} there is an integer $N$ such that
$\pi_{|U_j}^* : \Omega_{*} (X/G) \to \Omega_{*+\dim U_{j}}(X \times^G U_j)$ 
is an isomorphism for all $j > N$. 
\end{proof}

\begin{theorem}\label{propact}
Let $X$ be in $G-\mathbf{Sch}_k$ with proper $G$-action.  
Then $\pi^{*}:
\Omega_{*}(X/G)_{\mathbb{Q}} \to \Omega_{*+\dim G}^{G}(X)_{\mathbb{Q}}$
is an isomorphism.
\end{theorem}
This theorem implies that we have a ring structure on the cobordism
of a large class of interesting singular schemes. 
\begin{corollary}
Let $X$ be in $G-\mathbf{Sm}_k$.
Under the assumptions of the theorem, $\Omega^* (X/G)_{\mathbb{Q}}$ has
a ring structure.
\end{corollary}
\begin{proof}[Proof of Theorem \ref{propact}]
Write $g=\dim G$ and write $Y=X/G$ for the quotient. 
We proceed in a similar fashion as in 
\cite[Theorem 3(a)]{edd&gra-eit}. 
By \cite[Proposition 10]{edd&gra-eit}, we have a commutative diagram
\begin{equation*}
  \begin{CD} 
    X^{\prime} @>{\tilde{f}}>> X \\
    @V{p^{\prime}}VV @V{p}VV \\
    Y^{\prime} @>{f}>> Y,
  \end{CD}
\end{equation*}
where $X^{\prime} \to Y^{\prime}$ is a principal $G$-bundle and
both $X^{\prime}\to X$ and $Y^{\prime}\to Y$ are finite and
surjective.

We obtain from the exact sequences from Proposition \ref{exact} proved
below,   
\begin{gather*}
  \Omega_*^G(X^{\prime}\times_XX^{\prime})_{\mathbb{Q}}
  \xrightarrow{pr_{1 G} - pr_{2 G}} 
  \Omega_*^G(X^{\prime})_{\mathbb{Q}} \xrightarrow{\tilde{f}_G}
  \Omega_*^G(X)_{\mathbb{Q}} \to 0, 
\\  
  \Omega_*(Y^{\prime}\times_Y Y^{\prime})_{\mathbb{Q}}
    \xrightarrow{pr_{1*} - pr_{2*}} 
  \Omega_*(Y^{\prime})_{\mathbb{Q}} \xrightarrow{f_*}
  \Omega_*(Y)_{\mathbb{Q}} \to 0.  
\end{gather*}
Write $Y^{\prime \prime} = \left( X^{\prime} \times_X X^{\prime}
\right) /G$, we have then that
$\Omega_* \left( Y^{\prime \prime} \right) \to \Omega_* \left(
Y^{\prime} \times_Y Y^{\prime} \right)$ 
is a surjection since $Y'' \to Y'\times_YY'$ is a finite and
surjective morphism. 
We obtain the comparison of exact sequences, 
\begin{equation*}
\begin{CD}
  \Omega_* \left( Y ^{\prime \prime} \right)_{\mathbb{Q}} 
    @>{pr_{1*} - pr_{2*}}>>
  \Omega_* \left( Y^{\prime} \right)_{\mathbb{Q}} @>{f_*}>> 
  \Omega_*(Y)_{\mathbb{Q}} @>>> 0 \\ 
        @VV{\pi^{\prime \prime *}}V         @VV{\pi^{^{\prime}*}}V
        @VV{\pi^*}V @. \\ 
  \Omega_{*+g}^G \left( X^{\prime} \times_X X^{\prime}
    \right)_{\mathbb{Q}} @>{pr_{1 G} - pr_{2 G}}>>
  \Omega_{*+g}^G \left( X^{\prime} \right)_{\mathbb{Q}} @>{\tilde{f}_G}>>
    \Omega_{*+g}^G(X)_{\mathbb{Q}} @>>> 0.  
\end{CD}
\end{equation*}
By Proposition \ref{trivstab} the two maps on the left are
isomorphisms, so the third map too is.
\end{proof}

\begin{proposition}\label{exact}
Let $\pi:X^{\prime} \to X$ be a finite surjective map.
\begin{enumerate}
\item[(a)] The sequence
  $\; \Omega_* \left( X^{\prime} \times_X X^{\prime} \right)_{\mathbb{Q}} 
  \xrightarrow{p_{1*} - p_{2*}} 
  \Omega_* \left( X^{\prime} \right)_{\mathbb{Q}} 
  \xrightarrow{\pi_*} 
  \Omega_*(X)_{\mathbb{Q}} 
  \to 
  0 \;$
  is exact. 

\item[(b)] Suppose that both $X^{\prime}$ and $X$ are $G$-schemes and
  $\pi$ is $G$-equivariant. Then the sequence 
  $\; \Omega_*^G \left( X^{\prime} \times_X X^{\prime} \right)_{\mathbb{Q}}
  \xrightarrow{p^G_{1*} - p^G_{2*}} 
  \Omega_*^G \left( X^{\prime} \right)_{\mathbb{Q}} 
  \xrightarrow{\pi^G_*}
  \Omega_*^G(X)_{\mathbb{Q}} 
  \to 
  0 \;$  
  is exact.
\end{enumerate}
\end{proposition}

\begin{proof}
\begin{enumerate}
\item[(a)] By \cite[Proposition 1.8]{kim-fi&bt} the sequence
  \begin{equation*}
    {\rm CH}_*(X^{\prime} \times_X X^{\prime})_{\mathbb{Q}}
    \xrightarrow{p_{1*} - p_{2*}} 
    { \rm CH}_*(X^{\prime})_{\mathbb{Q}} \xrightarrow{\pi_*}
    {\rm CH}_*(X)_{\mathbb{Q}} \to 0,  
  \end{equation*}
  is exact. It follows that the sequence 
  \begin{equation*}
    {\rm CH}_*(X^{\prime} \times_X
    X^{\prime})[\mathbf{t}]^{(\mathbf{t})}_{\mathbb{Q}}
    \xrightarrow{p_{1*} - p_{2*}} {\rm
    CH}_*(X^{\prime})[\mathbf{t}]^{(\mathbf{t})}_{\mathbb{Q}}
    \xrightarrow{\pi_*} {\rm CH}
    _*(X)[\mathbf{t}]^{(\mathbf{t})}_{\mathbb{Q}} \to 0   
  \end{equation*}
  is also exact. 
  The statement follows from Example \ref{exa-tt}. 
  
\item[(b)]  The morphism
  $X^{\prime}\times^{G}U_i \to X\times^{G} U_i$ 
  is finite and surjective by faithfully flat descent, and 
  $\left( X^{\prime}\times^{G} U_i \right) \times_{X\times^{G} U_i}
  \left( X^{\prime} \times^{G} U_i \right)   
  = (X^{\prime}\times_XX^{\prime})\times^{G}U_i$. 
  Therefore 
  \begin{equation*}\label{ses-properaction}
    \Omega_* \left( \left( X^{\prime} \times_X X^{\prime} \right)
    \times^{G}U_i \right)_{\mathbb{Q}} 
    \xrightarrow{p_{1*} - p_{2*}} 
    \Omega_* \left( X^{\prime} \times^{G} U_i \right)_{\mathbb{Q}}
    \xrightarrow{\pi_*} 
    \Omega_* \left( X \times^{G} U_i \right)_{\mathbb{Q}} 
    \to 
    0
  \end{equation*}
  is an exact sequence for each $i$. 
  By Lemma \ref{ML} the system  
  $\; \Omega_* \left( \underbar{\;\;} \times^G U_i
  \right)_{\mathbb{Q}}$ 
  satisfies the Mittag-Leffler condition. 
  Hence the sequence remains exact upon taking inverse limits.
\end{enumerate}
\end{proof}

\subsection{Reduction to a Torus Action}
For the rest of the section $G$ will be a connected reductive linear
algebraic group containing a split
maximal torus $T$ and $B$ be a Borel subgroup containing $T$.  
Let $N$ be the normalizer of $T$ in $G$ and $W = N/T$ be the Weyl
group.  The Weyl group $W$ acts on $\Omega^{*}_{T}(X)$ and the image of the
restriction map is $W$-invariant and so we have the morphism
$\mathrm{res}_{G,T}:\Omega^{*}_{G}(X)\to \Omega^{*}_{T}(X)^{W}$.  
In \cite{edd&gra-eit} it is shown that the analogous morphism in
equivariant Chow groups is an isomorphism rationally. 
We show now that the analogous statement holds for equivariant
algebraic cobordism. 

\begin{lemma}\label{rta-lemma}
Let $X$ be in $B-\mathbf{Sch}$. 
Suppose that the principle $B$-bundle $X \to X/B$ exists. 
Then $\Omega^{*}(X/B) \to \Omega^{*}(X/T)$ is an isomorphism.
In particular for any $X$ in $B-\mathbf{Sch}$, the restriction map
$\mathrm{res}_{B,T} : \Omega_*^B (X) \to \Omega_*^T (X)$ is an
isomorphism.  
\end{lemma}
\begin{proof}
Since $B/T$ is an affine space, the map $X/T \to X/B$ is an affine
space bundle.  
The result follows from the extended homotopy axiom for cobordism. 
\end{proof}

Using the isomorphism in the previous lemma we transport the $W$-action on
$\Omega_{T}^{*}(X)$ to an action on $\Omega_{B}^{*}(X)$.  

\begin{lemma}
Let $G$ be a connected, reductive linear algebraic group. 
Then $\Omega^{*}(G/B)^{W}_{\Q} \cong \Omega^{*}(k)_{\Q}$.
\end{lemma}
\begin{proof}
Base change induces a $W$-equivariant isomorphism
$\Omega^{*}(G/B) \cong \Omega^{*}((G/B)_{L})$ for any field extension $L/k$ by Proposition \ref{invbase} because $G/B$ is cellular. Because $G$ and $T$ are determined by root
datum we may assume that $k=\C$. 
Specifically, if $k\subseteq \C$ then we have 
$\Omega^{*}(G/B)\iso\Omega^{*}((G/B)_{\C})$. 
If $\C\subseteq k$ then there is a split, connected reductive group $G'$ and Borel
subgroup $B'$ which are defined over $\C$ and are such that
$(G'/B')_{k} = G/B$ and so $\Omega^{*}(G/B) =
\Omega^{*}((G'/B')_{k})$.  

Consider the principal $W$-bundle $\pi:G/T\to G/N$. 
We can regard $\pi^{*}$ as the morphism
$\Omega^{*}(G/N)_{\Q}\to \Omega^{*}(G/T)_{\Q}^{W}$. 
Write $\pi'_{*}$ for the restriction of $\pi_{*}$ to
$\Omega^{*}(G/T)_{\Q}^{W}$. 
Since $\pi'_{*}\pi^{*}$ and $\pi^{*}\pi'_{*}$ coincide with
multiplication by $|W|$ we have that $\pi^{*}$ is an isomorphism.
Similarly we have the isomorphisms 
$\pi^{*} : \mathrm{CH}^{*} (G/N)_{\Q} \to
\mathrm{CH}^{*} (G/T)_{\Q}^{W} = \mathrm{CH}^{*} (G/B)_{\Q}^{W}$.  

The cycle map $\mathrm{CH}^{*} (G/B) \to \mathrm{H}^{*} (G/B)$
is a $W$-equivariant isomorphism \cite[Example 19.1.11]{ful-inter}. 
By \cite[Chapter III, $\S$1 (B)]{hsiang-book} we have 
$\mathrm{H}^{*} (G/B; \Q)^{W} = \Q$ and so 
$\mathrm{CH}^{*}( G/B)^{W}_{\C}$ from which it follows that
$\mathrm{CH}^{*} (G/N)_{\Q}= \Q$. 
Now from the isomorphism in Example \ref{exa-tt} we obtain the
isomorphism $\Omega^{*}(\C)_{\Q}=\Omega^{*}(G/N)_{\Q}$ and so  
$\Omega^{*}(\C)_{\Q} = \Omega^{*}(G/T)_{\Q}^{W}=\Omega^{*}(G/B)_{\Q}^{W}$.
\end{proof}

\begin{theorem}\label{the-loc}
Assume $G$ is a connected, reductive linear algebraic group.
Let $X$ be in $G-\mathbf{Sm}_k$. 
Then there is an isomorphism 
\[
  \mathrm{res}_{G,T}:\Omega^*_{G}(X)_{\Q} \to \Omega^*_{T}(X)_{\Q}^{W}
\]
of $\Omega^{*} (k)_{\mathbb{Q}}$-modules. 
If $G$ is special then 
$\mathrm{res}_{G,T}:\Omega^*_{G}(X) \to \Omega^*_{T}(X)^{W}$ is injective. 
\end{theorem}
\begin{proof}
 By Lemma \ref{rta-lemma} it suffices to show that 
$\Omega^*_{G}(X)_{\Q} \to \Omega^*_{B}(X)_{\Q}^{W}$ is an
isomorphism. By the following proposition we have  $\Omega^{*}(X
\times^G U_i)_{\Q}  \to \Omega^{*}(X \times^{B} U_i)_{\Q}^{W}$ is an
isomorphism for all $i$. 
Fixed points and inverse limits commute we are done.
\end{proof}

\begin{proposition}
Let $Y\to Q$ be a principle $G$-bundle, with $Y$ smooth. 
Write $p:Y/B\to Q$ for the resulting principal $G/B$-bundle. 
Then 
$$
p^{*}:\Omega^{*}(Q)_{\Q}\to \Omega^{*}(Y/B)_{\Q}^{W}
$$
is an isomorphism. If $G$ is special then $p^{*}:\Omega^{*}(Q)\to
\Omega^{*}(Y/B)^{W}$ is injective. 
\end{proposition}
\begin{proof}
By Proposition \ref{lhrat}, $\Omega^{*}(Y/B)_{\Q}$ is a
free $\Omega^{*}(Q)_{\Q}$-module. We may
take $1\in \Omega^{*}(Y/B)_{\Q}$ as one of the basis elements. 
It follows that $p^{*} :\Omega^{*}(Q)\to \Omega^{*}(Y/B)^{W}$ is
injective. 

Let $p:Y/B\to Q$ be obtained from a principal $G$-bundle $Y\to Q$ with
$Y$ not necessarily smooth. 
We proceed by induction to show that 
$p^{*}$ is surjective, the zero-dimensional case being done. 
We may assume that $Q$ is irreducible. 
Let $U\subseteq Q$ be an open subscheme over which the bundle is
isotrivial and let $Z$ be its closed complement. 
Consider the commutative diagram with exact rows
$$
\xymatrix{
\Omega_{*}(Z)_{\Q} \ar[r]\ar[d] & \Omega_{*}(Q)_{\Q} \ar[r]\ar[d] &
\Omega_{*}(U)_{\Q} \ar[d]\ar[r] & 0 \\ 
\Omega_{*}(Y_{Z})_{\Q}^{W} \ar[r] & \Omega_{*}(Y)_{\Q}^{W} \ar[r] &
\Omega_{*}(Y_{U})_{\Q}^{W} \ar[r]& 0, 
}
$$
where the bottom row is exact because taking fixed points is an exact
functor when $|W|$ is invertible. 
We see that it suffices to show that 
$\Omega_{*}(U)_{\Q}\to \Omega_{*}(Y_{U})_{\Q}^{W}$ is surjective. 
Let $g:V\to U$ be a finite \'etale morphism over which the bundle becomes
trivial. 
Consider the commutative square  
$$
\xymatrix{
\Omega_{*}(V)_{\Q} \ar[r]^{g_{*}}\ar[d] & \Omega_{*}(U)_{\Q}\ar[d]  \\
\Omega_{*}(V\times G/B)_{\Q}^{W} \ar[r] & \Omega_{*}(Y_{U})_{\Q}^{W} .
}
$$
The horizontal arrows are surjective and so we are reduced to the case
of the trivial bundle. 
In the proof of Proposition \ref{lhrat} it is shown that the map
$\Omega^{*}(V)_{\Q}\otimes_{\Omega^{*}(k)_{\Q}}\Omega^{*}(G/B)_{\Q}\to
\Omega^{*}(V\times G/B)_{\Q}$ induced by external product is a
surjection. 
This is an equivariant map and so surjectivity follows from the
isomorphism
\[
  (\Omega^{*}(V)_{\Q}\otimes_{\Omega^{*}(k)_{\Q}}\Omega^{*}(G/B)_{\Q})^{W}
=
  \Omega^{*}(V)_{\Q}\otimes_{\Omega^{*}(k)_{\Q}}\Omega^{*}(G/B)_{\Q}^{W}
=
  \Omega^{*}(V)_{\Q}.
\]  
If $G$ is special then $p^{*}$ is injective by Proposition \ref{lhrat}
and its image is contained in $\Omega^{*}(Y/B)^{W}$.  
\end{proof}

\begin{remark}\label{rem-gsp}
If $G$ is special, $T\subseteq G$ a maximal torus and we additionally
assume that the restriction map $\Omega^{*}_{T}(X) \to
\Omega^{*}_{T}(X^{T})$ is injective and $\Omega^{*}(X^{T})$ is
torsion-free then we can see that  
$\mathrm{res}_{G,T}:\Omega^{*}_{G}(X) \to \Omega^{*}_{T}(X)^{W}$ is an
isomorphism, which we explain below. 
In particular the case $X=\mathrm{spec}(k)$ shows that for $G$ special
we have the equality 
\[
  \Omega_G^{*}(k) = \Omega_T^{*}(k)^{W},
\]
giving the cobordism generalization of the result of Edidin-Graham for
integral Chow groups of classifying spaces of special groups
\cite{edd&gra-char}. 

Fix a good system $\{(V_{i},U_{i})\}$ of $G$-representations. 
Let $(x_{i})\in \Omega^{*}_{T}(X)^{W}$ be a fixed class. 
Write $p_{i}:X\times^{T}U_{i} \to X\times^{G}U_{i}$  for the projection. 
By Proposition \ref{lhrat} and Lemma \ref{rta-lemma} we have that
$\Omega^{*}(X\times^{T}U_{i})$ is a free
$\Omega^{*}(X\times^{G}U_{i})$-module. 
A basis is given by choosing elements which restrict to a basis of
$\Omega^{*}(G/T)$. 
We may choose a basis $e_{r,i}\in\Omega^{*}(X\times^{G}U_{i})$
inductively so that $e_{r,j}$ is mapped to $e_{r,i}$ under the
transition maps in the inverse system. Moreover we may choose each
$e_{1,i}=1$.  
Thus, we can write  $x_{i}= \sum_{k} p_{i}^{*}(y_{k,i})\cup e_{k,i}$ 
with $y_{k,i}\in \Omega^{*}(X\times^{G}U_{i})$.

By the previous proposition, for each $i$ there is an integer $M_{i}$
and a $y_i \in \Omega^{*}(X\times^{G}U_{i})$ so that 
$M_{i}x_{i} = p_{i}^{*}(y_i)$. 
Comparing the two expressions for $M_{i}x_{i}$ we see that 
$M_{i}x_{i} = M_{i} p_i^{*}(y_{1,i})$. 
Write $a_{i}=y_{1,i}$. 
Then for all $i$ the element
$x_{i}-p_{i}^{*}(a_{i})\in\Omega^{*}(X\times^{T}U_{i})$ is torsion. 
Note that the sequence $(a_{i})$ defines an element in
$\Omega^{*}_{G}(X)$. 
Hence we have a sequence $(x_{i}-p_{i}^{*}(a_{i}))$ of torsion
elements representing a class in  $\Omega^{*}_{T}(X)$. 

Write $c_{i}=x_{i}-p_{i}^{*}(a_{i})$. 
Since $\Omega^{*}_{T}(X)\subseteq \Omega^{*}_{T}(X^{T})$ we  also
write $(c_{i})\in \Omega^{*}_{T}(X^{T})$. 
Let $\{V'_{i}, U'_{i}\}$ be a good system of $T$-representations
with the property that $U_{i}/T$ is a product of projective spaces as in 
Example \ref{CRTA}. 
Consider the products $U_{i} \times U'_{j}$ and projections 
$f_{ij}: U_{i}\times U'_{j} \to U_{i}$ and 
$g_{ij}:U_{i}\times U'_{j} \to U'_{j}$. 
The maps of systems (indexed by pairs of integers)
$$
\{\Omega^{*}(X^{T}\times U_{i}/T)\} \xrightarrow{f_{ij}^{*}} 
\{ \Omega^{*}(X^{T}\times (U_{i}\times U'_{j})/T)\}
\xleftarrow{g_{ij}^{*}} \{\Omega^{*}(X^{T}\times U'_{j}/T)\} 
$$
induces an isomorphism on limits since each system computes
$\Omega^{*}_{T}(X^{T})$.  
Therefore there are $b_{j} \in \Omega^{*}(X^{T} \times U'_{j}/T)$ such
that $(f^{*}_{ij}(c_{i})) = (g^{*}_{ij}(b_{j}))$. 
The element $f^{*}_{ij}(c_{i})$ is torsion for all $i,j$.  
By Proposition \ref{prolim}, there are integers $C(i)$ and $D(j)$ such
that  $f^{*}_{ij}$ is an isomorphism for $j>C(i)$ and  $g_{ij}^{*}$ is
an isomorphism for $i>D(j)$.  
Since there is no torsion in $\Omega^{*}(X^{T}\times
U'_{j}/T)$ we must have that $f^{*}_{ik}(c_{i}) = 0$ for $i>D(k)$ but
then $f^{*}_{ij}(c_{i})= 0$ for $j>k$ as well. By taking $j>C(i)$
we conclude that $c_{i}=0$.  
\end{remark}


\section{Oriented Equivariant Borel-Moore Homology Theories}\label{ROT}

Our construction of equivariant algebraic cobordism
relies only on the general properties of an OBM theory and
the localization sequence. 
Thus, if $A_*$ is an OBM theory which has localization sequences
and $\{(V_{i},U_{i})\}$ is a good system of representations, we can
define the theory  
\[
  A_*^G (X) := \varprojlim_i A_* \left( X \times^G U_i \right)
\]
for any $X$ in $G-\mathbf{Sch}_k$. 
Similarly we define $A_{n}^G (X) = \varprojlim_i A_{n+\dim U_{i} -
  \dim G} ( X \times^G U_i )$. 
Because $A_{*}$ has  localization sequences the analogue of
Proposition \ref{prolim} is valid for $A_{*}$. 
This theory is then seen to be well-defined by reasoning as in Theorem
\ref{eacwelldef}.  

When the OCT $A^{*}$ classifies its formal group law, so that 
$\Omega^{*}(X)\otimes_{\Omega^{*}(k)}A^{*}(k) = A^{*}(X)$, define
\[
  A^{*}_{G}(X) 
= 
  \varprojlim_i A^* \left( X \times^G U_i \right)
\cong
  \Omega^{*}_{G}(X)\widehat{\otimes}_{\Omega^{*}(k)}A^{*}(k).
\]
Set $A^n_G (X) = \varprojlim_i A^n \left( X \times^G U_i \right)$.
Analogoues of all the computations and the results for equivariant
algebraic cobordism can be carried out for any such $A_*^G$.  

Define an 
\emph{oriented equivariant Borel-Moore theory} as a functor 
$A_*^G : G-\mathbf{Sch}'_k \to \mathbf{Ab}$
endowed with pull-backs maps for every $G$-equivariant l.c.i. morphism
satisfying the analogues of the properties listed in Sections 
\ref{sect-var}, \ref{sect-ls}, \ref{sect-pb} and \ref{sect-eh}. 

Similarly, an 
\emph{oriented equivariant cohomology theory} is a functor 
$A^*_G : (G-\mathbf{Sm}_k)^{op} \to \mathbf{Rng}$ 
endowed with morphisms $f_G : A^*_G (Y) \to A^{*+d}_G (X)$ of 
$A^*_G (k)$-modules for every projective morphism 
$f: Y \to X$ of relative dimension $d$ in $G-\mathbf{Sm}_k$, 
satisfying the analogues of the axioms listed in 
Section \ref{OCT}. 

\begin{remark}
If $A^G_*$ is an oriented equivariant Borel-Moore
theory arising from an OBM theory $A_*$, by construction we have a
canonical natural transformation $\Omega^G_* \to A^G_*$.
\end{remark}

\begin{example}
The universal additive OCT $\Omega^*_+$ induces the oriented
equivariant cohomology theory 
\[
  X \mapsto \Omega^*_{+ \, G} (X) = \varprojlim_i \Omega^*_+ 
  \left( X \times^G U_i \right).
\] 
Let $\mathrm{CH}_G^n$ be the $n$-th
equivariant Chow group as defined by Edidin and Graham 
\cite{edd&gra-eit}. 
The Chow group is the universal additive OCT theory (see 
Section \ref{UOCT}) and so there is $M_{n}$ such that 
\[
  \mathrm{CH}_G^n (X) 
:= 
  \mathrm{CH}^{n} \left( X \times^G U_i \right)
\cong
  \Omega_+^{n} \left( X \times^G U_i \right) 
\]
for any $i > M_{n}$. 
Since  $\mathrm{CH}^{n} (X \times^{G} U_{j}) \cong 
\mathrm{CH}^{n} ( X \times^{G} U_{i})$ for
$i,j >M_{n}$, by taking limits we get a natural isomorphism
$\Omega^{n}_{+ \, G}(X) \cong \mathrm{CH}^{n}_{G}(X)$.
This give us
\begin{equation*}
  \bigoplus_n \mathrm{CH}^{n}_{G}(X) 
= 
  \left( \bigoplus_n \Omega^{n}_{G}(X) \right)
  \widehat{\otimes}_{\Omega^{*}(k)} \Z. 
\end{equation*}
Also, from the remark above we obtain the commutative diagram
\[
  \xymatrix{
           \Omega^{n}_{G}(X) \ar[r]\ar[d]_{\mathrm{res}_{\Omega}^{G}}
           &  
           \mathrm{CH}_G^n (X)  \ar[d]^{\mathrm{res}}
             \\
           \Omega^{n}(X)\ar[r] &  \mathrm{CH}^n (X),
           }
\]
where the restriction map
$\mathrm{res} : \mathrm{CH}_G^n (X) \to \mathrm{CH}^n (X)$ 
is obtained by restricting to the fiber, as we did for 
$\mathrm{res}_{\Omega}^G$ in Section \ref{rest-fib}.
\end{example}

\begin{example}
The universal multiplicative periodic OCT $\Omega^{\times}_*$ induces
the oriented equivariant cohomology theory  
\[
  X \mapsto \Omega^*_{\times \, G} (X) = \varprojlim_i \Omega^*_{\times} 
  \left( X \times^G U_i \right).
\] 
Let $\mathrm{K}_G^0$ be the equivariant algebraic $K$-theory
defined by Thomason \cite{tho-ekt} as the Grothendieck group of the category
$G$-vector bundles (Section \ref{sect-gmod}). 
By descent the category of $G$-vector bundles on $X\times U_{i}$ is
equivalent to the category of vector bundles on $X \times^{G} U_{i}$. 

We thus have the induced natural isomorphism
$\mathrm{K}_{G}^{0} ( X \times U_{i} ) \cong 
\mathrm{K}^{0} ( X \times^{G} U_{i})$. 
Since $\mathrm{K} [\beta,\beta^{-1}]$ is the universal multiplicative
periodic OCT (see Section \ref{UOCT}), we can consider the composition
\[
  \mathrm{K}^0_G (X) 
\to
  \mathrm{K}^{0}( X \times^{G} U_{i})
\to 
  \mathrm{K}^{0}( X \times^{G} U_{i})[\beta,\beta^{-1}]
\cong
  \Omega^{*}_{\times}(X \times^G U_i).
\]
Taking limits defines the natural transformation
\[
  \mathrm{K}^0_G(X) \longrightarrow \Omega^*_{\times \, G}(X).
\]
Moreover, by forgetting the action we have a restriction map
$\mathrm{res}: \mathrm{K}^0_G \to \mathrm{K}^0$ which fits into a
commutative diagram 
\[
  \xymatrix{
            \mathrm{K}^0_G (X) \ar[r] \ar[d]_{\mathrm{res}} &
            \Omega^*_{\times \, G} (X) \ar[d]^{\mathrm{res}_G} \\
            \mathrm{K}^0 (X) \ar[r]_(.35){id \otimes 1} &
            \mathrm{K}^0 (X) [\beta,\beta^{-1}].
           }
\]
\end{example}


 \bibliography{referenc-eacs}
 \bibliographystyle{plain}

\end{document}